\documentclass[leqno]{amsart}
\usepackage{cases}
\usepackage{cases}
\usepackage{amsmath}
\usepackage{amsmath}
\usepackage{amsmath}
\usepackage{graphicx}
\usepackage{mathrsfs}
\usepackage{bbm}
\usepackage{amsfonts}
\usepackage{amssymb}
\usepackage{amsmath,amsthm}
\usepackage{color}
\usepackage{array,siunitx}
\usepackage{mathrsfs}
\usepackage{stmaryrd}
\usepackage{amsthm,amsmath,amssymb}
\usepackage{booktabs}
\usepackage{mathrsfs}
\usepackage{stmaryrd}
\usepackage[backref]{hyperref}
\usepackage{geometry} 
\usepackage{algorithm,algorithmic}
\usepackage{amsmath,amsfonts,amsthm,bm} 
\usepackage{float} \usepackage{comment}
\usepackage{color}
\usepackage{graphicx} 
\newcommand{\normmm}[1]{{\left\vert\kern-0.25ex\left\vert\kern-0.25ex\left\vert #1 
		\right\vert\kern-0.25ex\right\vert\kern-0.25ex\right\vert}}
\usepackage{lipsum}
\makeatletter
\numberwithin{equation}{section}
\newtheorem{Theorem}{Theorem}[section]

\newtheorem{Lemma}[Theorem]{Lemma}
\newtheorem{Corollary}[Theorem]{Corollary}

\newtheorem{Remark}[Theorem]{Remark}

\newtheorem{Definition}[Theorem]{Definition}

\theoremstyle{definition}

\makeatother
\title[Tracking Dirichlet data for Bernoulli free boundary problems]{On the well-posedness of tracking Dirichlet data for Bernoulli free boundary problems}

\author{WEI GONG AND LE LIU}
\thanks{$^*$School of Mathematical Sciences, University of  Chinese Academy of Sciences \& LSEC, Institute of Computational Mathematics, Academy of Mathematics and Systems Science, Chinese Academy of Sciences, Beijing 100190, China. Email: wgong@lsec.cc.ac.cn; liule2020@lsec.cc.ac.cn. The authors acknowledge the support from the National Key Research and
Development Program of China (project no. 2022YFA1004402) and
the National Natural Science Foundation of China (project no. 12071468).}

\begin{document}
	\maketitle
	\begin{abstract}
The aim of this paper is to study the shape optimization method for solving the Bernoulli free boundary problem, a well-known ill-posed problem that seeks the unknown free boundary through Cauchy data. Different formulations have been proposed in the literature that differ in the choice of the objective functional. Specifically, it was shown respectively in \cite{eppler2010tracking1} and  \cite{eppler2010tracking2} that tracking Neumann data is well-posed but tracking Dirichlet data is not. In this paper we propose a new well-posed objective functional that tracks Dirichlet data at the free boundary. By calculating the
 Euler derivative and the shape Hessian of the objective functional we show that the new formulation is well-posed, i.e., the shape Hessian is coercive at the minima. The coercivity of the shape Hessian may ensure the existence of optimal solutions for the nonlinear Ritz-Galerkin approximation method and its convergence, thus is crucial for the formulation. As a summary, we conclude that tracking Dirichlet or Neumann data in their energy norm is not sufficient, but tracking them in a half an order higher norm will be well-posed. To support our theoretical results we carry out extensive numerical experiments.
	\end{abstract}
 

\section{introduction}

In this paper we study the so-called exterior Bernoulli free boundary problem that is described below: Given sufficient smooth functions $f \geqslant 0$, $g>0$ and $h>0$, find a domain $\Omega$ (or free boundary $\Gamma$) that satisfies the following over-determined boundary value problem:
\begin{equation*}
  (\mathbf{P} )\quad
    \begin{cases}
        -\Delta u=f \quad {\rm in}\quad \Omega ,\\
        -\frac{\partial u}{\partial \textbf{n}}=g,\ u=0 \quad {\rm on}\quad \Gamma, \\
        u=h\quad {\rm on}\quad\Sigma,
        \end{cases}
\end{equation*}
where the inner boundary $\Sigma$  of the domain $\Omega$ is fixed and the outer boundary $\Gamma$ is the free boundary, we refer to Fig. \ref{fig:c06h06_trim} for an illustration.

\begin{figure*}[!htbp]
    \centering
    \includegraphics[trim = 3mm 4mm 3mm 3mm, clip, width=0.40\textwidth]{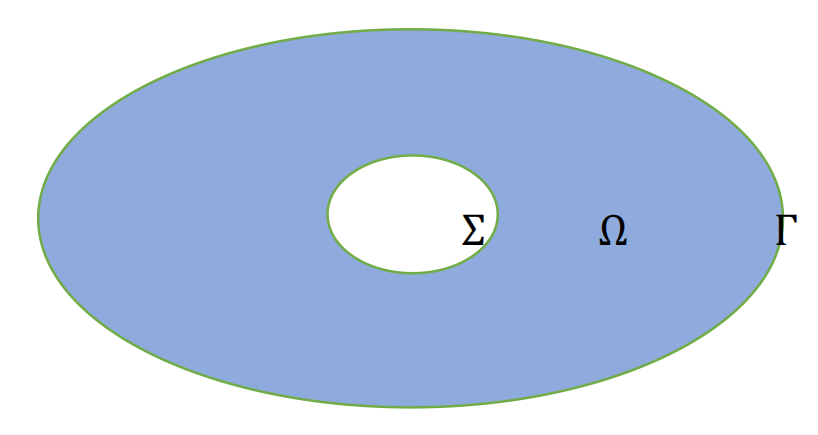}
    \caption{An illustration of the domain $\Omega$, its fixed boundary $\Sigma$ and free boundary $\Gamma$.}
    \label{fig:c06h06_trim}
\end{figure*}

Bernoulli free boundary problems arise in many applications, including the ideal fluid dynamics, optimal design, electro chemistry, electro statics, to name a few. In this paper, we do not study the existence of a solution to problem $(\mathbf{P})$. Instead, we assume that there exists a solution to problem $(\mathbf{P})$ and denote it by $\Omega^*$, and denote by $\Gamma^*$ the free boundary of $\Omega^*$. For the existence of solutions we refer to  \cite{caffarelli1981existence} for more details.

There are already some classical methods for solving Bernoulli free boundary problems, including the variational methods, the implicit Neumann scheme and so on, more details can be found in \cite{bach1999variational}. Shape optimization approach is among one of the most popular numerical methods for solving free boundary problems (cf. \cite{Haslinger,Brugger, Burman}). In \cite{eppler2010tracking1}, Eppler and Harbrecht proposed to track Neumann data at the free boundary $\Gamma$ to solve problem $(\mathbf{P})$, that is, to solve the following shape optimization problem:
\begin{equation}
    \inf_{\Omega } J_1(\Omega )=\frac{1}{2}\int_{\Gamma} \left(g+\frac{\partial u}{\partial \textbf{n}}\right)^2 d\sigma,
    \label{con:obj1}    
\end{equation}
    subject to
    \begin{equation}
        \begin{cases}
        -\Delta u=f \quad {\rm in}\quad \Omega, \\
        u=0 \quad {\rm on}\quad \Gamma, \\
        u=h \quad {\rm on}\quad\Sigma. 
        \end{cases}
    \label{con:j1}
\end{equation}
On the other hand, in \cite{eppler2010tracking2} the authors proposed to track Dirichlet data at the free boundary $\Gamma$, that is, 
\begin{equation}
    \inf_{\Omega } J_2(\Omega )=\frac{1}{2}\int_{\Gamma} u^2 d\sigma,
    \label{con:obj2}
\end{equation}
 subject to
    \begin{equation}
        \begin{cases}
        -\Delta u=f \quad {\rm in}\quad \Omega, \\
        -\frac{\partial u}{\partial \textbf{n}}=g \quad {\rm on}\quad \Gamma, \\
        u=h \quad {\rm on}\quad\Sigma .
        \end{cases}
        \label{con:j2}
\end{equation}

Obviously, if there exists $\Omega^*$ (or $\Gamma^*$) that is a solution to problem $(\mathbf{P})$, then $\Omega^*$ must be the optimal solution to the shape optimization problem (\ref{con:obj1})–(\ref{con:j1}) (and $J_1(\Omega^*)=0$); conversely, if $J_1(\Omega^*)=0$, then $\Omega^*$ is also a solution to problem $(\mathbf{P})$. The same property holds for the objective functional $J_2$.

When using the shape optimization method to solve the problem $(\mathbf{P})$, the objective functional $J$ should be chosen to satisfy the following property:
\begin{itemize}
    \item (SP): $J$ is a non-negative functional, i.e. $J(\Omega) \geq 0$ for all $\Omega$;
$J(\Omega^*)=0$ if and only if $\Omega^*$ is the solution to the problem $(\mathbf{P})$.
\end{itemize}
Obviously, $J_1$ and $J_2$ satisfy the property (SP).

Shape optimization problem is generally strongly nonconvex, so we usually expect only local minima. The convergence of shape optimization algorithms to the local minimum depends not only on the choice of the initial guess, but also on the local property of the minimum, i.e., the second order sufficient  optimality condition to ensure the local optimality. However, the latter depends strongly on the formulation of the shape optimization problem, or in other words, on the choice of the objective functional. Taking the Bernoulli free boundary problem as an example, we prefer to choose the objective functional such that its value is strictly greater than $0$ in the neighborhood of $\Omega^*$ except at $\Omega^*$, that is, to ensure the local uniqueness of the optimal solution $\Omega^*$ of the shape optimization problem. We call such an objective functional well-posed, or the corresponding shape optimization problem is well-posed.

The well-posedness of shape optimization problems is rarely studied in the literature (cf. \cite{EpplerCC,eppler2010tracking1,eppler2010tracking2}). However, it is indispensable for the convergence analysis of shape optimization algorithms. In \cite{Kiniger} and \cite{Fumagalli} a two-dimensional shape optimization problem
with the portion of the boundary to be optimized being the graph of a function
was studied for elliptic and Stokes equations, respectively. The shape optimization problem was transformed into an optimal control problem and second-order convergence of the numerical approximations to a local solution of the optimization problem was proved under the second-order
sufficient optimality condition. The boundary parametrization of an elliptic shape optimization problem was considered in \cite{eppler2007convergence}, where error estimates for a finite element method (FEM) were obtained
under the assumption that the optimal domain is star-shaped and the infinite dimensional
shape optimization problem admits a stable optimizer satisfying the
second-order optimality condition. 

For a systematic study of the well-posedness of shape optimization problems we refer to the work of Eppler and Harbrecht (cf. \cite{eppler2010tracking1,eppler2010tracking2,eppler2006efficient,eppler2012kohn}). They considered the Euler derivative and shape Hessian for shape optimization problems posed on star-shaped domains, and studied the well-posedness of shape functionals by analyzing their coercivity at the minima. For star-shaped domains the free boundary can be represented by a function $r\in X:=C^{k,\alpha}(\mathbb{S} ^{n-1})$, where $\mathbb{S} ^{n-1}$ is the unit sphere in $\mathbb{R} ^{n}$. Specifically,
\begin{equation*}
    \Gamma =\left\{x=r(\widehat{x} )\widehat{x}:\quad \widehat{x}\in \mathbb{S} ^{n-1}  \right\} .
\end{equation*}
Let $dr\in X$ be a function that can generate a velocity field that transforms $\Gamma$ into the perturbed one
\begin{equation*}
    \Gamma_{\varepsilon }=\left\{x=r(\widehat{x} )\widehat{x}+\varepsilon dr(\widehat{x})\widehat{x}:\quad \widehat{x}\in \mathbb{S} ^{n-1}  \right\} .
\end{equation*} 
Similarly, when studying the shape Hessian, the associated boundary is given by
\begin{equation*}
     \Gamma_{\varepsilon_1, \varepsilon_2}=\left\{x=r(\widehat{x} )\widehat{x}+\varepsilon_1dr_1(\widehat{x})\widehat{x}+\varepsilon_2dr_2(\widehat{x})\widehat{x}:\quad \widehat{x}\in \mathbb{S} ^{n-1}  \right\} ,    
\end{equation*}
where $dr_1,dr_2\in X$.

For the problem (\ref{con:obj1})-(\ref{con:j1}) that tracks Neumann data, Eppler and Harbrecht pointed out in \cite{eppler2010tracking1} that the first-order optimality condition $dJ_1(\Omega^*)[dr]=0$ holds for all $dr \in X$ at the solution $\Omega^*$ of the Bernoulli free boundary problem, and the shape Hessian is a bilinear continuous functional on $H^1(\Gamma)\times H^1(\Gamma)$ in the neighborhood $B_\delta ^X(r^*)$ of $\Omega^*$, that is,
\begin{equation*}
    \left\lvert d^2J_1(\Omega)[dr_1,dr_2]\right\rvert 
\leq c(r)\left\lVert dr_1\right\rVert _{H^1(\Gamma)}
\left\lVert dr_2\right\rVert _{H^1(\Gamma)}
\quad \forall r \in B_\delta ^X(r^*),
\end{equation*}
and the corresponding second-order Taylor remainder  $R_2(J_1(\Omega),dr)$ satisfies
\begin{equation*}
    \left\lvert R_2(J_1(\Omega),dr) \right\rvert 
    =o(\left\lVert dr\right\rVert _{C^{3,\alpha}(\Gamma)})
    \left\lVert dr\right\rVert _{H^1(\Gamma)}^2.
\end{equation*}
Then it was shown in \cite{eppler2007convergence} that if the coercivity condition
\begin{equation*}
d^2J_1(\Omega^*)[dr,dr]\geq c_E\left\lVert dr\right\rVert _{H^1(\Gamma)}^2\quad \forall dr \in X
\end{equation*}
is satisfied, then the objective functional $J_1$ is well-posed and the existence of optimal solutions and the convergence of the nonlinear Ritz-Galerkin method are guaranteed. The authors did prove in \cite{eppler2010tracking1} that the coercivity holds for the objective functional $J_1$.

For the problem (\ref{con:obj2})-(\ref{con:j2}) that tracks Dirichlet data, Eppler and Harbrecht  showed in \cite{eppler2010tracking2} that the shape Hessian is also a bilinear continuous functional on $H^1(\Gamma)\times H^1(\Gamma)$ in the neighborhood $B_\delta ^X(r^*)$ of $\Omega^*$,
but they only proved that
\begin{equation*}
    d^2J_2(\Omega^*)[dr,dr]\geq c_E\left\lVert dr\right\rVert _{L^2(\Gamma^*)}^2\quad \forall dr \in X ,
\end{equation*}
so we cannot determine whether $J_2$ is well-posed. It was further shown in \cite{eppler2010tracking2} that $J_2$ is algebraically ill-posed.

In addition, Eppler and Harbrecht proposed in 
\cite{eppler2006efficient} the following 
shape optimization problem for tracking Neumann data:
\begin{equation}
    \inf_{\Omega } J_3(\Omega )=\int_{\Omega}(\left\lVert \nabla u\right\rVert^2-2fu+g^2)dx ,
\end{equation}
 subject to
\begin{equation}
        \begin{cases}
        -\Delta u=f \quad {\rm in}\quad \Omega ,\\
        u=0 \quad {\rm on}\quad \Gamma, \\
        u=h \quad {\rm on}\quad\Sigma .
        \end{cases}
\end{equation}
They proved that the objective functional $J_3$ is also well-posed, where the shape Hessian is a bilinear continuous functional defined on $H^{1/2}(\Gamma)\times H^{1/2}(\Gamma)$, and
\begin{equation*}
    d^2J_3(\Omega^*)[dr,dr]\geq c_E\left\lVert dr\right\rVert _{H^{1/2}(\Gamma^*)}^2\quad \forall dr \in X .
\end{equation*}

It is worth mentioning that Kohn and Vogelius proposed the following shape optimization problem \cite{leugering2012constrained,kohn1984determining}:
\begin{equation}\label{J4_obj}
    \inf_{\Omega } J_4(\Omega )=\int_{\Omega}\left\lVert \nabla (v-w)\right\rVert^2 dx,
\end{equation}
subject to
\begin{equation}\label{J4_state}
        \begin{cases}
        -\Delta v=f \quad {\rm in}\quad \Omega, \\
        v=0 \quad {\rm on}\quad \Gamma, \\
        v=h \quad {\rm on}\quad\Sigma
        \end{cases}\quad
        \text{and}\quad
        \begin{cases}
            -\Delta w=f \quad {\rm in}\quad \Omega, \\
            -\frac{\partial w}{\partial \textbf{n}}=g \quad {\rm on}\quad\Gamma, \\
            w=0 \quad {\rm on}\quad \Sigma.
        \end{cases}
\end{equation}
Eppler and Harbrecht proved in \cite{eppler2012kohn} that 
$J_4(\Omega) \sim
\left\lVert w\right\rVert _{H^{1/2}(\Gamma)} \left\lVert g+\frac{\partial v}{\partial n} \right\rVert _{H^{-1/2}(\Gamma)}$. Therefore, this problem simultaneously tracks Dirichlet and Neumann data at the free boundary. They pointed out that the shape Hessian is a bilinear continuous functional on $H^1(\Gamma)\times H^1(\Gamma)$ in the neighborhood $B_\delta ^X(r^*)$ of $\Omega^*$, 
but they only proved that
\begin{equation*}
d^2J_4(\Omega^*)[dr,dr]\geq c_E\left\lVert dr\right\rVert _{H^{1/2}(\Gamma^*)}^2 \quad \forall dr \in X,
\end{equation*}
so again we can not determine whether $J_4$ is well-posed, see \cite{eppler2012kohn} for more details.

In the past decades, several other objective functionals have also been proposed to solve Bernoulli free boundary problems (cf. \cite{Boulkhemair,Rabago}). However, to the best of our knowledge, there is no well-posed objective functional for tracking Dirichlet data at the free boundary. In this paper we propose a new objective functional and consider the following shape optimization problem:
\begin{equation}\label{con:jnew}
    \inf_{\Omega } J(\Omega )=\frac{1}{2}\int_{\Gamma}\left(\frac{\partial w}{\partial \textbf{n}}\right)^2 d\sigma ,
\end{equation}
subject to
\begin{equation}
        \begin{cases}
        -\Delta u=f \quad {\rm in}\quad \Omega, \\
        -\frac{\partial u}{\partial \textbf{n}}=g \quad {\rm on}\quad \Gamma ,\\
        u=h \quad {\rm on}\quad\Sigma 
        \end{cases}\quad
        \text{and}\quad
        \begin{cases}
            -\Delta w=0 \quad {\rm in}\quad \Omega ,\\
            w=u \quad {\rm on}\quad\Gamma, \\
            w=0 \quad {\rm on}\quad \Sigma.
        \end{cases}
        \label{con:equ1d6}
\end{equation}
Obviously, this objective functional tracks Dirichlet data at the free boundary and satisfies the property (SP). We show that this objective functional is well-posed by analysing the coercivity of its shape Hessian. The coercitivy of the shape Hessian is crucial for the formulation. In fact, it will ensure the existence of optimal solutions for the nonlinear Ritz-Galerkin approximation method and its convergence (cf. \cite{eppler2007convergence}).

The structure of this article is arranged as follows: In Section 2 we will give relevant knowledge of shape optimizations, and give also the optimality theorem of shape optimization problems.  In Section 3, we will compute the Euler derivative and the shape Hessian of our proposed new objective functional, and show the well-posedness of the shape optimization problem by proving the coercivity of the shape Hessian. In Section 4 we give the motivation to choose the new objective functional, some discussions are also presented.  In Section 5, we will carry out some numerical experiments to support our theory.

\section{Preliminaries}

This section is divided into two parts. On the one hand, we introduce some basic knowledge of shape optimizations, including the velocity field, the transformation of domains, the material and shape derivatives, the Euler derivative and shape Hessian. This part mainly refers to \cite{delfour2011shapes,sokolowski1992introduction}. On the other hand, we will give the optimality theorem of shape optimization problems, we refer to \cite{eppler2007convergence} for more details.

\subsection{Basic knowledge of shape optimizations}

Shape calculus concerns the variation of the objective functional with respect to the perturbation of the domain. There are two common approaches to achieve domain perturbations, one is the velocity method, and the other is the direct transformation method. In this paper we adopt the former one.

$V$ is called a velocity field, if for $\tau >0$, $V:[0,\tau]\times \mathbb{R}^n\rightarrow \mathbb{R}$ satisfies 
    \begin{gather*}
    \forall x \in \mathbb{R}^n, { }V(\cdot,x) \in C([0,\tau];\mathbb{R}^n),\\
    \exists c>0,\forall x,y \in \mathbb{R}^n, { }\left\lVert V(\cdot,y)-V(\cdot,x) \right\rVert _{C([0,\tau];\mathbb{R}^n)} 
    \leq c\left\lvert y-x\right\rvert.
    \end{gather*}
In the following, $V(0)$ represents the vector-valued function $V(0,x)$.
The transformation of the domain can be realized by the velocity field. In fact,
let $x_V(t;X)$ denote the solution of the differential equation
\begin{equation}
    \begin{cases}
        \frac{dx}{dt}(t)=V(t,x(t)) \quad t \in [0,\tau] ,\\
        x(0)=X \in \Omega,
    \end{cases}
\end{equation}
then the transformation can be defined as follows:
\begin{equation*}
    X \mapsto T_t(V)(X)\triangleq x_V(t;X) :\Omega \rightarrow \Omega _t.
\end{equation*}
Sometimes $\Omega _t$ is also written as $\Omega_t(V)$. We refer to Fig. \ref{fig:qybh} for an illustration of the transformation of domains.

\begin{figure}[!htbp]
    \centering
    \includegraphics[trim = 3mm 4mm 3mm 3mm, clip, width=0.8\textwidth]{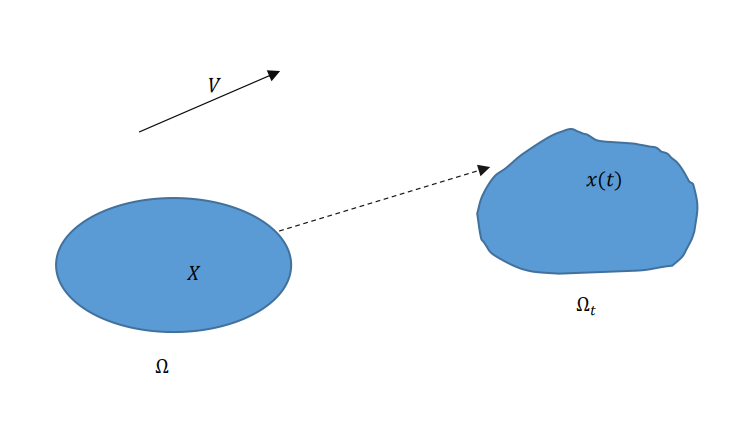}
    \caption{The transformation of domains.}
    \label{fig:qybh}
\end{figure}


Now we can introduce the definitions of the Euler derivative and shape Hessian. 
\begin{Definition}
    Let $\tau >0$, $V$ is a velocity field, $J$ is a real-valued shape functional, if the limit
    \begin{equation*}
    \lim_{t \to 0^+} \frac{J(\Omega _t)-J(\Omega)}{t} 
\end{equation*}
    exists, it is called the Euler derivative of $J$ at $\Omega$ along the velocity field $V$, denoted by $dJ(\Omega)[V]$. Let $W$ be another velocity field, we know that $dJ(\Omega)[V]$ is also a real-valued shape functional, if the limit
\begin{equation*}
    \lim_{t \to 0^+} \frac{dJ(\Omega_t(W))[V]-dJ(\Omega)[V]}{t} 
\end{equation*}
exists, it is called the shape Hessian
of $J$ at $\Omega$ along the velocity fields $V$ and $W$, denoted by $d^2J(\Omega)[V,W]$.
\end{Definition}

For functions defined from $\mathbb{R}^n$ to $\mathbb{R}$ we can define the material and shape derivatives, for which we distinguish the domain-type and the boundary-type \cite[Chapter 2, pp. 98-114]{sokolowski1992introduction}.

\begin{Definition}
    Let $\Omega$ be a bounded domain with a $C^k$ boundary $\partial \Omega$  and let $V$ be a velocity field. For $y(\Omega)\in W^{s,p}(\Omega)$, $s\in [0,k]$, $p \in [1,\infty)$, if the limit
\begin{equation*}
    \lim_{t \to 0^+} \frac{y(\Omega _t)\circ T_t(V)-y(\Omega)}{t}
\end{equation*}
exists, it is called the material derivative of $y$ at $\Omega$ along the velocity field $V$, denoted by $\overset{\circ }{y}(\Omega;V)\in W^{s,p}(\Omega)$. For $y(\Gamma)\in W^{r,p}(\Gamma)$, $r\in [0,k]$, $p \in [1,\infty)$,
if the limit
\begin{equation*}
    \lim_{t \to 0^+} \frac{y(\Gamma _t)\circ T_t(V)-y(\Gamma)}{t}
\end{equation*}
exists, it is called the material derivative of $y(\Gamma)$ along the velocity field $V$, denoted by $\overset{\circ}{y}(\Gamma;V)\in W^{r,p}(\Gamma)$.
\end{Definition}

The definition of shape derivatives will be given by the material derivative.
\begin{Definition}
      Let $\Omega$ be a bounded domain with a $C^k$ boundary $\partial \Omega=\Gamma$. Let $V$ be a velocity field and $W(\Omega)$ be a Sobolev space.
      If $y(\Omega) \in W(\Omega)$ has a material derivative $\overset{\circ }{y}(\Omega;V)$, 
      then 
      $y'(\Omega;V)=\overset{\circ}{y}(\Omega;V)-\nabla y(\Omega)\cdot V(0)\in W(\Omega)$
      is called the shape derivative of $y(\Omega)$ along the velocity field $V$. Let $W(\Gamma)$ be another Sobolev space.
      If $z(\Gamma) \in W(\Gamma)$ has a material derivative $\overset{\circ }{z}(\Gamma;V)$, 
      then 
      $z'(\Gamma;V)=\overset{\circ}{z}(\Omega;V)-\nabla_{\Gamma} z(\Gamma)\cdot V(0)\in W(\Gamma)$
      is called the shape derivative of $z(\Gamma)$ along the velocity field $V$.
\end{Definition}

The shape derivative of the solution of partial differential equations plays an important role in the following analysis.
The proof of the following two theorems can be found in  
\cite[Chapter 3, pp. 118-121]{sokolowski1992introduction}.
\begin{Theorem}\label{thm:dl22}
    Let $h(\Omega) \in L^2(\Omega)$, $z(\Gamma) \in H^{1/2}(\Gamma)$, and 
\begin{equation}
    \begin{cases}
        -\Delta y(\Omega)=h(\Omega) \quad {\rm in}\quad L^2(\Omega), \\
        y(\Omega)=z(\Gamma) \quad {\rm on}\quad\Gamma,
    \end{cases}
\end{equation}
then $y'(\Omega;V)\in H^1(\Omega)$ and
\begin{equation}
    \begin{cases}
        -\Delta y'(\Omega;V)=h'(\Omega;V) \quad {\rm in}\quad \mathcal{D} '(\Omega), \\
        y'(\Omega;V)=-\frac{\partial y}{\partial \textbf{n}}(\Omega)V(0)\cdot \textbf{n}+z'(\Gamma;V) \quad {\rm on}\quad\Gamma.
    \end{cases}
\end{equation}
\end{Theorem}
\begin{Theorem}\label{thm:dl23}
    Let $h(\Omega) \in L^2(\Omega)$, $z(\Gamma) \in H^{1/2}(\Gamma)/\mathbb{R}$,
  which  satisfy the compatibility condition:
    \begin{equation*}
        \int_{\Omega}h(\Omega)dx+\int_{\Gamma}z(\Gamma)d\sigma=0
    \end{equation*}
and 
\begin{equation}
    \begin{cases}
        -\Delta y(\Omega)=h(\Omega) \quad {\rm in}\quad L^2(\Omega) ,\\
        \frac{\partial y(\Omega)}{\partial \textbf{n}}=z(\Gamma) \quad {\rm on}\quad\Gamma,
    \end{cases}
\end{equation}
then $y'(\Omega;V)\in H^1(\Omega)/\mathbb{R} $ and
\begin{equation}
    \begin{cases}
        -\Delta y'(\Omega;V)=h'(\Omega;V) \quad {\rm in}\quad \mathcal{D} '(\Omega) ,\\
        \frac{\partial y'(\Omega;V)}{\partial \textbf{n}}=
        {\rm div}_{\Gamma}(V(0)\cdot  \textbf{n}\nabla_\Gamma y(\Omega))
        +[h(\Omega)+\mathcal{H} z(\Gamma)]V(0)\cdot \textbf{n}+
        z'(\Gamma;V) \quad {\rm on}\quad\Gamma,
    \end{cases}
\end{equation}
where $\mathcal{H}$ is the additive curvature of $\Gamma$, i.e.,  $n-1$ times of the mean curvature of $\Gamma$.
\end{Theorem}

\subsection{The optimality theorem for shape optimization problems}
In this subsection we present the optimality theorem for shape optimization problems. For simplicity, we only consider problems defined in the star-shaped domains (cf. \cite[Chapter 4, pp. 177-178]{delfour2011shapes} and \cite{eppler2010tracking1}).

Let $\mathbb{S} ^{n-1}$ be the unit sphere in $n$-dimensional Euclid space and  
\begin{equation}
\Omega=\left\{
x=\rho \widehat{x} \in \mathbb{R} ^{n}:\quad \widehat{x}\in 
\mathbb{S} ^{n-1},\quad L(\widehat{x} )\leq \rho \leq U(\widehat{x} )  \right\} .
\nonumber
\end{equation}
Herein, $L \in X$ is the parametrization of the fixed boundary and $U \in X$ is the parametrization of the free boundary. In order to ensure the smoothness of the boundary, we take $X=C^{3,\alpha}( \mathbb{S} ^{n-1})$. Unless otherwise specified, $X$ in the following text has the same meaning.
It can be seen that the free boundary of the domain $\Omega$ is $\Gamma=\{U(\widehat{x})\widehat{x}:\widehat{x}\in  \mathbb{S}^{n-1}\}$. Since the inner boundary $\Sigma=\{L(\widehat{x})\widehat{x}:\widehat{x}\in  \mathbb{S}^{n-1}\}$ is fixed, the domain $\Omega$ corresponds one-to-one with the free boundary $\Gamma$, and thus the domain $\Omega$ corresponds one-to-one with the function $U$. Therefore, we can directly use $U$ to refer to the region $\Omega$. We can also directly represent $J(\Omega)$ as $J(U)$.

We consider the transformed domain given by:
\begin{equation}
    \Omega_{t}=\left\{
    x=\rho \widehat{x} \in \mathbb{R} ^{n}:\quad \widehat{x}\in 
    \mathbb{S} ^{n-1},\quad L(\widehat{x} )\leq \rho \leq U(\widehat{x} ) 
    +tdr(\widehat{x} ) \right\} ,
    \nonumber
\end{equation}
where $dr\in X$,
then the corresponding point of $x=\rho \widehat{x} \in \Omega$ in $\Omega_t$ is given by
\begin{equation}
    T_t(x )=x
    +t\left[ \frac{\rho-L(\widehat{x} )}{U(\widehat{x} )
    -L(\widehat{x} )}\right] dr(\widehat{x} )\widehat{x}.
    \nonumber
\end{equation}
Therefore, according to the relationship between the velocity field and the transformation \cite[Chapter 4, pp. 181-183]{delfour2011shapes}, we can obtain $V(t,x)=\frac{\partial T_t(T^{-1}_t(x))}{\partial t}$, and particularly $V(0,x)|_{\Gamma}=dr(\widehat{x})\widehat{x}$. We call $V$ the velocity field induced by $dr$, and directly denote $V$ as $dr$. Correspondingly, we can also denote $dJ(\Omega)[V]$ as $dJ(U)[dr]$.

For the shape optimization problem
\begin{equation}
    \inf_{r \in X} J(r),
\end{equation}
we have the following theorem:
\begin{Theorem}
{\rm{(\cite[pp. 9-11]{eppler2007convergence})}}
Let
\begin{itemize}
    \item[(A1)] $dJ(r^*)[dr]=0$, $\forall dr \in X$ holds for a certain $r^*\in X$;
    
\item[(A2)] there exists $\delta>0$, such that
\begin{equation*}
    \left\lvert d^2J(r)[h_1,h_2]\right\rvert \leq c_S(r)\left\lVert h_1\right\rVert _{H^s}\left\lVert h_2\right\rVert _{H^s}
    \quad \forall h_1,h_2 \in H^s
\end{equation*}
holds for all $r \in \overline{B_{\delta}^X(r^*)} $.
Herein, $H^s$ is a Sobolev space, and $X\subset H^s$.
\end{itemize}

Then the domain $r^*$ is a strong regular local optimum of second order, i.e., 
\begin{equation*}
    J(r)-J(r^*)\geq \frac{c_E}{4}\left\lVert r-r^*\right\rVert ^2_{H^s}
    \quad \forall r \in B_{\delta}^X(r^*)
\end{equation*}
if and only if the following two conditions hold:
\begin{itemize}
    \item[(A3)] the shape Hessian is strongly coercive at $r^*$:
\begin{equation*}
    d^2J(r^*)[h,h]\geq c_E\left\lVert h\right\rVert _{H^s}^2
   \quad \forall h \in H^s.
\end{equation*}

\item[(A4)] the following estimate
\begin{equation*}
    \left\lvert 
        d^2J(r)[h_1,h_2]
        -d^2J(r^*)[h_1,h_2]
    \right\rvert 
    \leq 
    \eta (\left\lVert r-r^*\right\rVert _{X} )
    \left\lVert h_1\right\rVert _{H^s}
    \left\lVert h_2\right\rVert _{H^s}
    \quad \forall h_1,h_2 \in H^s
\end{equation*}
holds for all $r \in B_{\delta}^X(r^*)$,
where $\eta$ satisfies $\lim_{t \to 0^+}\eta(t)=0.$
\end{itemize}

\label{thm:dl24}
\end{Theorem}
\begin{proof}
    Here we only prove sufficiency. In fact, by Taylor's expansion, we have
   \begin{equation*}
    J(r)-J(r^*)=0+\frac{1}{2}d^2J(r^*+\xi h)[h,h]\quad \xi \in (0,1)
\end{equation*}
    holds for all $r=r^*+h\in B_{\delta}^X(r^*)$.
    Therefore, by conditions $(A3)$ and $(A4)$, we can obtain
    \begin{equation*}
    \begin{split}
        J(r)-J(r^*)
        &=\frac{1}{2}d^2J(r^*)[h,h]+\frac{1}{2}d^2J(r^*+\xi h)[h,h]-\frac{1}{2}d^2J(r^*)[h,h]\\
        &\geq \frac{c_E}{2}\left\lVert h\right\rVert _{H^s}^2
        -\frac{1}{2}\eta (\left\lVert \xi h \right\rVert )\left\lVert h\right\rVert _{H^s}^2\\
        &=\left(\frac{c_E}{2}-\frac{1}{2}\eta (\left\lVert \xi h \right\rVert )\right) \left\lVert r-r^*\right\rVert _{H^s}^2.
    \end{split}
\end{equation*}
Thus, we arrive at the conclusion in view of $\lim_{t \to 0^+}\eta(t)=0$.
\end{proof}

\begin{Remark}
A few remarks are in order:
\begin{itemize}
    \item[(1)] When proving sufficiency of the above theorem, condition $(A3)$ can be weakened to
    \begin{equation*}
     (A3')\quad d^2J(r^*)[h,h]\geq c_E\left\lVert h\right\rVert _{H^s}^2
   \quad \forall h \in X.
   \end{equation*}

\item[(2)] The condition $(A2)$ holds for
$ h_1,h_2 \in H^s$ in the dense sense, which can be referred to \cite[pp. 4-5]{eppler2007convergence}. For boundary-type objective functionals, it holds for $s=1$, we refer to \cite{eppler2012kohn,eppler2000optimal} and \cite[p. 5]{eppler2007convergence} for the details.

\item[(3)] The condition $(A4)$ ensures that the second-order Taylor remainder $R_2(J(\Omega),dr)$ satisfies
 \begin{equation*}
        \left\lvert R_2(J(\Omega),dr) \right\rvert 
        =o(\left\lVert dr\right\rVert _{X})
        \left\lVert dr\right\rVert _{H^s}^2.
    \end{equation*}
    For the boundary-type objective functional studied in this paper, condition $(A4)$ holds for $s=1$, see \cite{dambrine2000stability,dambrine2002variations} for details.
\end{itemize}
\end{Remark}

Condition $(A3')$ plays a crucial role for the existence of solutions and convergence of the nonlinear Ritz-Galerkin approximation method. We recall the following theorem from \cite[Part II, pp. 283-284]{leugering2012constrained}.
\begin{Theorem}
    Assume that the Euler derivative satisfies condition $(A1)$ and the shape Hessian satisfies conditions $(A2)$, $(A3')$ at $r^*$, then there exists a neighborhood $U(r^*)\subset X$ of  $r^*$ such that the discrete problem
    \begin{equation}
    seek \quad r^*_N \in V_N\quad s.t. \quad dJ(r_N^*)[dr]=0\quad\forall dr \in V_N
    \label{con:1j}
\end{equation}
has a unique solution $r^*_N \in V_N \cap U(r^*)$, and we have an estimate
\begin{equation}
    \left\lVert r^*_N-r^* \right\rVert _{H^s}
    \lesssim 
    \underset{r_N \in V_N}{\inf}\left\lVert r_N-r^*\right\rVert _{H^s},
\end{equation}
where $V_N={\rm span} \{\varphi _1,\varphi _2,...,\varphi _N\}\subset X$ with $N$ sufficiently large.
\label{thm:DL214}
\end{Theorem}

\begin{Remark}
Since $V_N$ is a linear space and the shape Hessian is continuous, the first-order necessary condition (\ref{con:1j}) will indeed ensure that $r_N^*$ is a local optimal solution, along the same line with the sufficiency proof of Theorem \ref{thm:dl24}.
\end{Remark}

\section{Shape calculus}
In this section we calculate the Euler derivative and shape Hessian for the newly proposed shape optimization problem (\ref{con:jnew})-(\ref{con:equ1d6}). The motivations and discussions on this objective functional will be postponed to Section 4. 

\subsection{Several lemmas}
In this subsection we first collect some necessary materials for performing the shape calculus. For the vector filed $V$ we will use the abbreviation $v_n:=\left\langle V(0),\textbf{n}\right\rangle$ with $\textbf{n}$ being the unit outer normal vector of $\Gamma$.
\begin{Lemma}\label{thm:yinli3}
Let $y(\Omega)\in H^1(\Omega)$ be a function defined in the domain $\Omega$, $z(\Gamma)$ be a function defined on the free boundary $\Gamma$, $V$ be a velocity field, and let the boundary-type shape derivative $z'(\Gamma;V)$ of $z(\Gamma)$ and the domain-type shape derivative $y'(\Omega;V)$ of $y(\Omega)$ both exist. If $z(\Gamma)=y(\Omega)|_{\Gamma}$, then the following identity holds: \begin{equation} 
z’(\Gamma;V)=y’(\Omega;V) +\frac{\partial y(\Omega)}{\partial \textbf{n}}v_n\quad {\rm on} \quad \Gamma.
\end{equation} 
\end{Lemma}
\begin{proof}
It follows from the definitions of the material and shape derivatives that on $\Gamma$ it holds 
\begin{equation*}
\begin{split}
    z'(\Gamma;V)
    &=\overset{\circ}{z} (\Gamma;V)-\nabla _{\Gamma}z(\Gamma)\cdot V(0)\\
    &=\overset{\circ}{y} (\Omega;V)-\nabla _{\Gamma}y(\Omega)\cdot V(0)\\
    &=y'(\Omega;V)+\nabla y(\Omega) \cdot V(0)-\nabla _{\Gamma}y(\Omega)\cdot V(0)\\
    &=y'(\Omega;V)+\frac{\partial y(\Omega)}{\partial \textbf{n}}v_n.
\end{split}
\end{equation*}
This gives the result.
\end{proof}

\begin{Lemma}{\rm{(\cite[Chapter 3, pp. 125-126]{sokolowski1992introduction}, \cite[Chapter 9, pp. 488-491]{delfour2011shapes})}}\label{thm:yinli4}
    Let $\textbf{n}$ be the unit outer normal vector of the free boundary and let $N=\nabla b$ be its extension in $\Omega$. Then the following identities hold:
    \begin{gather*}
        \overset{\circ}{\textbf{n}} (\Gamma;V)=\overset{\circ}{N} (\Omega;V)=(DV\textbf{n}\cdot \textbf{n})\textbf{n}- {}^*DV\textbf{n}, \quad {\rm on} \quad \Gamma, \\
        N'(\Omega;V)=(DV\textbf{n}\cdot \textbf{n})\textbf{n}- {}^*DV\textbf{n}-\nabla^2bV ,\nonumber
    \end{gather*}
    where $b(x)=d_\Omega(x)-d_{\Omega^c}(x)$, $d$ is the distance function.
We often denote $\overset{\circ}{\textbf{n}}(\Gamma;V)$ as
    $d\textbf{n}[V]$.    
\end{Lemma}

\begin{Lemma}
   Let $y(\Omega)\in H^2(\Omega)$, and let $z(\Gamma)=\frac{\partial y(\Omega)}{\partial \textbf{n}}$. Then the following identity holds:
   \begin{equation}
    z'(\Gamma;V)=\frac{\partial dy}{\partial n}+
    \nabla y\cdot (d\textbf{n}-\nabla^2bV)+\frac{\partial ^2y}{\partial \textbf{n}^2}v_n\quad \rm on \quad \Gamma,
   \end{equation}  
   where $dy=y'(\Omega;V)$, $\frac{\partial ^2y}{\partial \textbf{n}^2}=\nabla^2y\textbf{n} \cdot \textbf{n}$, $d\textbf{n}=\overset{\circ}{\textbf{n}} (\Gamma;V)$.
   \label{thm:yinli5}
\end{Lemma}
\begin{proof}
In fact, we know that $z(\Gamma)=\nabla y \cdot \nabla b |_{\Gamma}$.
By Lemma \ref{thm:yinli3} we have
\begin{equation*}
    z'(\Gamma;V)=(\nabla y \cdot \nabla b)'(\Omega;V)|_{\Gamma}
    +\frac{\partial (\nabla y \cdot \nabla b) }{\partial \textbf{n}}v_n.
\end{equation*}
Using chain's rule we can obtain
\begin{equation*}
\begin{split}
    (\nabla y \cdot \nabla b)'(\Omega;V)
=\nabla dy \cdot \nabla b +\nabla y \cdot N'(\Omega;V)
=\nabla dy \cdot \nabla b +\nabla y \cdot (d\textbf{n}-\nabla^2bV).
\end{split}
\end{equation*}
Notice that $\nabla^2b \textbf{n}=0$, there holds
\begin{equation*}
    \begin{split}
    \frac{\partial (\nabla y \cdot \nabla b) }{\partial \textbf{n}}
    =\nabla(\nabla y \cdot \nabla b) \cdot \textbf{n}
    =(\nabla^2y\nabla b+\nabla^2b\nabla y)\cdot \textbf{n} 
    =\nabla^2y\textbf{n} \cdot \textbf{n}
    \end{split}
\end{equation*}
on $\Gamma$.
Therefore, the conclusion follows.
\end{proof}

\begin{Lemma}{\rm{(\cite[Chapter 2, pp. 115-116]{sokolowski1992introduction})}}\label{thm:yinli6}
Let $z(\Gamma )\in L^{1}(\Gamma )$ have a shape derivative $z'(\Gamma ;V) \in L^{1}(\Gamma )$, then the Euler derivative of the shape functional $J(\Omega)=\int_\Gamma z(\Gamma )d\sigma$ reads  
\begin{equation}
    dJ(\Omega )[V]=\int_\Gamma z'(\Gamma;V)d\sigma +
    \int_\Gamma {\rm div}_\Gamma (z(\Gamma) V(0))d\sigma.
\end{equation}
\end{Lemma}


\begin{Lemma}{\rm{(\cite[Chapter 9, p. 497]{delfour2011shapes})}}    \label{thm:yinli8}
Let $\textbf{v}\in  C^1(\Omega)^n $,  $v_n=\left\langle \textbf{v},\textbf{n}\right\rangle $ and $v_{\Gamma}=\textbf{v}-v_n\textbf{n}$, then
the following identities hold:
\begin{gather}
D\textbf{v}|_{\Gamma}=D_{\Gamma}\textbf{v}+D\textbf{v}\textbf{n}{}^*\textbf{n},\quad
{\rm div}_{\Gamma}\textbf{v}={\rm div}_{\Gamma}v_{\Gamma}+\mathcal{H}v_n,\quad 
\nabla _{\Gamma}v_n={}^*D_{\Gamma}\textbf{v}\textbf{n}+\nabla^2bv_{\Gamma}.\nonumber
\end{gather}
\end{Lemma}

\begin{Lemma}{\rm{(\cite[Chapter 2, p. 94]{sokolowski1992introduction})}}
Let $\Gamma$ belong to $C^2$ and $\phi \in H^3(\Omega)$,
then the following identity holds:
\begin{equation*}
    \Delta \phi=\Delta _{\Gamma}\phi +\mathcal{H} \frac{\partial \phi}{\partial \textbf{n}}
    +\frac{\partial ^2\phi}{\partial \textbf{n}^2}\quad {\rm on} \quad \Gamma,
\end{equation*}
where $\frac{\partial ^2\phi}{\partial \textbf{n}^2}=\nabla^2\phi \textbf{n} \cdot \textbf{n}$.
    \label{thm:yinli9}
\end{Lemma}

Finally, by Theorems \ref{thm:dl22} and \ref{thm:dl23}, we can derive the following lemmas.
\begin{Lemma}\label{thm:yinli10}
Let $u'(\Omega,V)$ be the shape derivative of $u$ in (\ref{con:equ1d6}) along the velocity field $V$, and we denote $u'(\Omega,V)$ as $du$, then $du$ satisfies the following equation:
\begin{equation}
    \begin{cases}
    -\Delta du=0\quad {\rm in}\quad \Omega ,\\
    \frac{\partial du}{\partial \textbf{n}}={\rm div}_{\Gamma }(v_n\nabla _{\Gamma }u)+v_n\left[f-\mathcal{H}  g-\frac{\partial g}{\partial \textbf{n}}\right] \quad {\rm on}\quad \Gamma, \\
    du=0\quad {\rm on}\quad\Sigma, 
    \end{cases}
\end{equation}
where $\mathcal{H}$ is the additive curvature of $\Gamma$.
\end{Lemma}
\begin{Lemma}\label{thm:yinli11}
    Let $w'(\Omega,V)$ be the shape derivative of $w$ in (\ref{con:equ1d6}) along the velocity field $V$, and we denote $w'(\Omega,V)$ as $dw$, then $dw$ satisfies the following equation:
 \begin{equation}
        \begin{cases}
        -\Delta dw=0\quad {\rm in}\quad \Omega \\
        dw=-\frac{\partial w}{\partial \textbf{n}}v_n+du-gv_n\quad {\rm on}\quad \Gamma \\
        dw=0\quad {\rm on}\quad\Sigma.
        \end{cases}
    \end{equation}
\end{Lemma}

\subsection{Euler derivative}
Now we are ready to calculate the Euler derivative. 
\begin{Theorem}
    The Euler derivative of $J(\Omega )=\frac{1}{2}\int_{\Gamma}\left(\frac{\partial w}{\partial \textbf{n}}\right)^2 d\sigma$
    at $\Omega\subset \mathbb{R}^{n}$ along the velocity field $V$ reads
\begin{equation}
dJ(\Omega)[V]=
\int_{\Gamma}v_n\left\{-\frac{1}{2}p^2\mathcal{H} 
+p\left[ f-\mathcal{H} g-\frac{\partial g}{\partial \textbf{n}} \right] 
-(g+p)\frac{\partial p}{\partial \textbf{n}}\right\} d\sigma,\label{J_Euler}
\end{equation}
where  $p$ satisfies the following adjoint equation:
\begin{equation}
    \begin{cases}
        -\Delta p=0\quad {\rm in}\quad \Omega \\
        p=\frac{\partial w}{\partial \textbf{n}}\quad {\rm on}\quad \Gamma ,\\
        p=0\quad {\rm on}\quad\Sigma.
        \end{cases}\label{adjoint_state}
\end{equation}
\end{Theorem}
\begin{proof}
We choose $z(\Gamma)=\frac{1}{2}\left(\frac{\partial w}{\partial \textbf{n}}\right)^2 $ in Lemma \ref{thm:yinli6} and denote $V(0)$ as $V$, then
\begin{equation}
    \begin{split}
        dJ(\Omega )[V]
        &=\int_\Gamma z'(\Gamma;V)d\sigma +
        \int_\Gamma {\rm div}_\Gamma (z(\Gamma) V)d\sigma \\
        &=\int_\Gamma z'(\Gamma;V)d\sigma +
        \int_\Gamma \nabla_\Gamma z(\Gamma) \cdot Vd\sigma 
        +\int_\Gamma z(\Gamma){\rm div}_{\Gamma}Vd\sigma.
    \end{split}\nonumber
\end{equation}
Using Lemma \ref{thm:yinli5} we have
\begin{equation*}
    \left(\frac{\partial w}{\partial \textbf{n}}\right)'(\Gamma;V)
    =\frac{\partial dw}{\partial \textbf{n}}+
    \nabla w\cdot (d\textbf{n}-\nabla^2bV)+\frac{\partial ^2w}{\partial \textbf{n}^2}v_n.
\end{equation*}
By using chain's rule, we obtain that
\begin{equation*}
    \begin{split}
        z'(\Gamma;V)
        &=\frac{\partial w}{\partial \textbf{n}}
        \left(\frac{\partial dw}{\partial \textbf{n}}+
        \nabla w\cdot (d\textbf{n}-\nabla^2bV)+\frac{\partial ^2w}{\partial \textbf{n}^2}v_n\right).
    \end{split}
\end{equation*}
Noticing that
\begin{gather*}
 \frac{1}{2}\nabla[\left(\nabla w\cdot \nabla b\right)^2 ]
 =\left(\nabla w\cdot \nabla b\right)\left(\nabla^2w\nabla b+\nabla^2b\nabla w\right) 
 =\frac{\partial w}{\partial \textbf{n}}\left(\nabla^2w\textbf{n}+\nabla^2b\nabla w\right)
\end{gather*}
 and $\nabla^2b\textbf{n}=\textbf{0}$ on $\Gamma$,
 we get
\begin{gather*}
    {1\over 2}\frac{\partial (\nabla w\cdot \nabla b)^2}{\partial \textbf{n}}=
\frac{\partial w}{\partial \textbf{n}}\frac{\partial ^2w}{\partial \textbf{n}^2}.
\end{gather*}
It follows from the definition of the tangential gradient that
\begin{equation*}
\nabla _{\Gamma}z(\Gamma)=
\frac{\partial w}{\partial \textbf{n}}\left(\nabla^2w\textbf{n}+\nabla^2b\nabla w\right)
-\frac{\partial w}{\partial \textbf{n}}\frac{\partial ^2w}{\partial \textbf{n}^2}\textbf{n}.
\end{equation*}

Moreover, it holds
\begin{equation*}
    \nabla _{\Gamma}z(\Gamma)\cdot V=
    \frac{\partial w}{\partial \textbf{n}}\left(\nabla^2w\textbf{n}\cdot V+\nabla^2b\nabla w\cdot V\right)
    -\frac{\partial w}{\partial \textbf{n}}\frac{\partial ^2w}{\partial \textbf{n}^2}v_n.
\end{equation*}
It follows from Lemma \ref{thm:yinli8} that
\begin{equation}
    z(\Gamma){\rm div}_{\Gamma}V=\frac{1}{2}\left(\frac{\partial w}{\partial \textbf{n}}\right)^2 
    \left[{\rm div}_{\Gamma}v_{\Gamma}+\mathcal{H}  v_n\right] \nonumber,
\end{equation}
we can obtain that
\begin{equation*}
    \begin{split}
    dJ(\Omega)[V]
    &=\int_{\Gamma} \frac{\partial w}{\partial \textbf{n}}
    \left[\frac{\partial dw}{\partial \textbf{n}}+
    \nabla w\cdot (d\textbf{n}-\nabla^2bV)\right] \\
    &+\frac{\partial w}{\partial \textbf{n}}\left(\nabla^2w\textbf{n}\cdot V+\nabla^2b\nabla w\cdot V\right)+\frac{1}{2}\left(\frac{\partial w}{\partial \textbf{n}}\right)^2 
    \left[{\rm div}_{\Gamma}v_{\Gamma}+\mathcal{H}  v_n\right] d\sigma.
    \end{split}
\end{equation*}
Next, we will simplify the expression of $dJ(\Omega)[V]$ in a few steps.

To begin with, we recall the tangential Green's formula (cf. \cite[Chapter 9, p. 498]{delfour2011shapes}): 
\begin{equation}
    \int_\Gamma f{\rm div}_{\Gamma}\textbf{v}+\nabla_{\Gamma}f\cdot \textbf{v} d\sigma
    =\int_\Gamma \mathcal{H}  f\textbf{v}\cdot \textbf{n}d\sigma\label{tangential_green}
\end{equation}
for $f \in C^1(\Gamma)$ and $\textbf{v} \in C^1(\Gamma)^n$. Therefore, we obtain
\begin{equation*}
    \int_{\Gamma}\frac{1}{2}\left(\frac{\partial w}{\partial \textbf{n}}\right)^2
    {\rm div}_{\Gamma}v_{\Gamma}d\sigma
    =-\frac{1}{2}\int_{\Gamma}\nabla_{\Gamma} \left(\left(\frac{\partial w}{\partial \textbf{n}}\right)^2 \right) \cdot v_{\Gamma}d\sigma.
\end{equation*}
It follows from the expression of $\nabla_{\Gamma}z(\Gamma)$ that
\begin{equation*}
    \nabla_{\Gamma}  \left(\left(\frac{\partial w}{\partial \textbf{n}}\right)^2 \right) 
    =
2\left[\frac{\partial w}{\partial \textbf{n}}\left(\nabla^2w\textbf{n}+\nabla^2b\nabla w\right)
-\frac{\partial w}{\partial \textbf{n}}\frac{\partial ^2w}{\partial \textbf{n}^2}\textbf{n}\right] .
\end{equation*}
Moreover, we conclude from the identity $v_{\Gamma}\cdot \textbf{n}=0$ that
\begin{equation*}
    \nabla_{\Gamma}  \left(\left(\frac{\partial w}{\partial \textbf{n}}\right)^2 \right) \cdot v_{\Gamma}
    =2\left[\frac{\partial w}{\partial \textbf{n}}\left(\nabla^2w\textbf{n}+\nabla^2b\nabla w\right)\cdot v_{\Gamma} \right]. 
\end{equation*}
As a result, 
\begin{equation*}
    \int_{\Gamma}\frac{1}{2}\left(\frac{\partial w}{\partial \textbf{n}}\right)^2
    {\rm div}_{\Gamma}v_{\Gamma}d\sigma
    =
    -\int_{\Gamma} \frac{\partial w}{\partial \textbf{n}}\left(\nabla^2w\textbf{n}\cdot v_{\Gamma}+\nabla^2b\nabla w\cdot v_{\Gamma}\right)
    d\sigma .
\end{equation*}
Combining the above computations, we obtain
\begin{equation*}
    \begin{split}
    &\int_{\Gamma}\frac{1}{2}\left(\frac{\partial w}{\partial \textbf{n}}\right)^2
    {\rm div}_{\Gamma}v_{\Gamma}
    +\frac{\partial w}{\partial \textbf{n}}\left(\nabla^2w\textbf{n}\cdot V+\nabla^2b\nabla w\cdot V\right)\\
    &=\int_{\Gamma}\frac{\partial w}{\partial \textbf{n}}
    \left(
        \nabla^2w\textbf{n}\cdot V+\nabla^2b\nabla w\cdot V
        -\nabla^2w\textbf{n}\cdot v_{\Gamma}-\nabla^2b\nabla w\cdot v_{\Gamma}
    \right) d\sigma\\
    &=\int_{\Gamma}\frac{\partial w}{\partial \textbf{n}}
    \left(\frac{\partial ^2w}{\partial \textbf{n}^2}v_n+v_n\nabla^2b\nabla w\cdot \textbf{n}\right) d\sigma\\
    &=\int_{\Gamma}\frac{\partial w}{\partial \textbf{n}}\frac{\partial ^2w}{\partial \textbf{n}^2}v_n d\sigma.
    \end{split}
\end{equation*}

Secondly, we have
\begin{equation}\nonumber
    \begin{split}
        \frac{\partial w}{\partial \textbf{n}}\nabla w\cdot\left[d\textbf{n}-\nabla^2bV\right]
    &=\frac{\partial w}{\partial \textbf{n}}{}^{*}\nabla w\left[d\textbf{n}-\nabla^2bV\right]\\
    &=\frac{\partial w}{\partial \textbf{n}}{}^{*}\nabla w\left[(DV\textbf{n}\cdot \textbf{n})\textbf{n}- {}^*DV\textbf{n}-\nabla^2bV\right]\\
    &=\frac{\partial w}{\partial \textbf{n}}\left[(DV\textbf{n}\cdot \textbf{n}){}^{*}\nabla w\textbf{n}
    -\nabla w \cdot \left({}^*DV\textbf{n}+\nabla^2bV\right)   \right] \\
    &=\frac{\partial w}{\partial \textbf{n}}\left[\frac{\partial w}{\partial \textbf{n}}(DV\textbf{n}\cdot \textbf{n})
    -\nabla w \cdot \left({}^*DV\textbf{n}+\nabla^2bV\right)   \right] .
    \end{split}
\end{equation}
By using Lemma \ref{thm:yinli8} we have
\begin{equation*}
{}^{*}DV\textbf{n}={}^{*}D_{\Gamma}V\textbf{n}+
\left(DV\textbf{n} \cdot \textbf{n}\right)\textbf{n} .
\end{equation*}
Therefore,
\begin{equation*}
\nabla w \cdot \left({}^{*}DV\textbf{n}\right) 
=\nabla w \cdot \left({}^{*}D_{\Gamma}V\textbf{n}\right)
+\frac{\partial w}{\partial \textbf{n}}\left(DV\textbf{n} \cdot \textbf{n}\right) .
\end{equation*}
Moreover, we can deduce from  Lemma \ref{thm:yinli4}, Lemma \ref{thm:yinli8} and the tangential Green's formula that
\begin{equation}   \nonumber
\begin{split}
\frac{\partial w}{\partial \textbf{n}}\nabla w\cdot\left[d\textbf{n}-\nabla^2bV\right]
=-\frac{\partial w}{\partial \textbf{n}}
\nabla w \cdot \left[{}^{*}D_{\Gamma}V\textbf{n}+\nabla^2bV\right] 
=-\frac{\partial w}{\partial \textbf{n}}
\nabla w \cdot \nabla _{\Gamma}v_n
=-\frac{\partial w}{\partial \textbf{n}}
\nabla  _{\Gamma}w \cdot \nabla _{\Gamma}v_n
\end{split}
\end{equation}
and 
\begin{equation*}
    \begin{split}
    \int_{\Gamma}\frac{\partial w}{\partial \textbf{n}}\nabla w\cdot\left[d\textbf{n}-\nabla^2bV\right]d\sigma
    =-\int_{\Gamma}\frac{\partial w}{\partial \textbf{n}}
    \nabla  _{\Gamma}w \cdot \nabla _{\Gamma}v_nd\sigma
    =\int_{\Gamma}{\rm div}_{\Gamma}\left(\frac{\partial w}{\partial \textbf{n}}
    \nabla  _{\Gamma}w\right) v_nd\sigma.
    \end{split}
\end{equation*}
So far, we can obtain
\begin{equation*}
    \begin{split}
        dJ(\Omega)[V]
        =\int_{\Gamma}\frac{\partial w}{\partial \textbf{n}}
        \frac{\partial dw}{\partial \textbf{n}}
        +\left[
            {\rm div}_{\Gamma}\left(\frac{\partial w}{\partial \textbf{n}}
            \nabla  _{\Gamma}w\right)
            +\frac{\partial w}{\partial \textbf{n}}\frac{\partial ^2w}{\partial \textbf{n}^2}
        +\frac{1}{2}\left(\frac{\partial w}{\partial \textbf{n}}\right)^2\mathcal{H} 
            \right]v_n d\sigma.
    \end{split}
\end{equation*}

Next, let $p$ satisfy the adjoint equation defined in (\ref{adjoint_state}). 
By using Green's formula, the tangential Green's formula (\ref{tangential_green}) and the equations of $w,du,dw$ on $\Gamma$, we can derive that
\begin{equation*}
\begin{split}
    \int_{\Gamma}\frac{\partial w}{\partial \textbf{n}}
    \frac{\partial dw}{\partial \textbf{n}}d\sigma
    &=\int_{\Gamma}p\frac{\partial dw}{\partial \textbf{n}}d\sigma
    =\int_{\Gamma}dw\frac{\partial p}{\partial \textbf{n}}d\sigma\\
    &=\int_{\Gamma}du\frac{\partial p}{\partial \textbf{n}}
    -\left(g+\frac{\partial w}{\partial \textbf{n}}\right)\frac{\partial p}{\partial \textbf{n}}v_n d\sigma
    =\int_{\Gamma}p\frac{\partial du}{\partial \textbf{n}}
    -\left(g+\frac{\partial w}{\partial \textbf{n}}\right)\frac{\partial p}{\partial \textbf{n}}v_n d\sigma\\
    &=\int_{\Gamma}p{\rm div}_{\Gamma}(v_n\nabla_{\Gamma}u)d\sigma
    +\int_{\Gamma}p\left[f-\mathcal{H} g-\frac{\partial g}{\partial \textbf{n}}\right] v_nd\sigma
    -\left(g+\frac{\partial w}{\partial \textbf{n}}\right)\frac{\partial p}{\partial \textbf{n}}v_n d\sigma\\
    &=\int_{\Gamma}v_n\left\{-\nabla_{\Gamma}p\cdot \nabla_{\Gamma}u
    +p\left[f-\mathcal{H} g-\frac{\partial g}{\partial \textbf{n}}\right] 
    -\left(g+\frac{\partial w}{\partial \textbf{n}}\right)\frac{\partial p}{\partial \textbf{n}}
    \right\} d\sigma\\
    &=\int_{\Gamma}v_n\left\{-\nabla_{\Gamma}p\cdot \nabla_{\Gamma}w
    +p\left[f-\mathcal{H} g-\frac{\partial g}{\partial \textbf{n}}\right] 
    -\left(g+p\right)\frac{\partial p}{\partial \textbf{n}}
    \right\} d\sigma.
\end{split}
\end{equation*}
This implies
\begin{equation*}
    \begin{split}
        dJ(\Omega)[V]
       & =\int_{\Gamma}
        \left\{-\nabla_{\Gamma}p\cdot \nabla_{\Gamma}w
        +p\left[f-\mathcal{H} g-\frac{\partial g}{\partial \textbf{n}}\right] 
        -\left(g+p\right)\frac{\partial p}{\partial \textbf{n}}
        \right\}v_n 
        \\
       &\quad+\left[
            {\rm div}_{\Gamma}\left(\frac{\partial w}{\partial \textbf{n}}
            \nabla  _{\Gamma}w\right)
            +\frac{\partial w}{\partial \textbf{n}}\frac{\partial ^2w}{\partial \textbf{n}^2}
        +\frac{1}{2}\left(\frac{\partial w}{\partial \textbf{n}}\right)^2\mathcal{H} 
            \right]v_n d\sigma\\
        &=\int_{\Gamma}
        \left\{-\nabla_{\Gamma}p\cdot \nabla_{\Gamma}w
        +p\left[f-\mathcal{H} g-\frac{\partial g}{\partial \textbf{n}}\right] 
        -\left(g+p\right)\frac{\partial p}{\partial \textbf{n}}
        \right\}v_n 
        \\
       & \quad+\left[
            {\rm div}_{\Gamma}\left(p
            \nabla  _{\Gamma}w\right)
            +p\frac{\partial ^2w}{\partial \textbf{n}^2}
        +\frac{1}{2}p^2\mathcal{H} 
            \right]v_n d\sigma.
    \end{split}
\end{equation*}

Finally, using Lemma \ref{thm:yinli9} we get
\begin{equation*}
    0=\Delta w=\frac{\partial ^2w}{\partial \textbf{n}^2}+\Delta_{\Gamma}w
    +\mathcal{H} \frac{\partial w}{\partial \textbf{n}}.
\end{equation*}
Therefore,
\begin{equation*}
    \begin{split}
        {\rm div}_{\Gamma}\left(p\nabla  _{\Gamma}w\right)
        &=\nabla_{\Gamma}p\cdot \nabla _{\Gamma}w 
        +p\Delta_{\Gamma}w\\
        &=\nabla_{\Gamma}p\cdot \nabla _{\Gamma}w 
        +p\left(-\frac{\partial ^2w}{\partial \textbf{n}^2}-
        \mathcal{H} \frac{\partial w}{\partial \textbf{n}}\right)\\
        &=\nabla_{\Gamma}p\cdot \nabla _{\Gamma}w 
        -p\frac{\partial ^2w}{\partial \textbf{n}^2}-
        p^2\mathcal{H}  .
    \end{split}
\end{equation*}
We now arrive at 
\begin{equation*}
    \begin{split}
        dJ(\Omega)[V]
        &=\int_{\Gamma}v_n\left\{
-\frac{1}{2}p^2\mathcal{H} 
+p\left[f-\mathcal{H} g-\frac{\partial g}{\partial \textbf{n}}\right]
-\left(g+p\right)\frac{\partial p}{\partial \textbf{n}}
        \right\} d\sigma.
    \end{split}
\end{equation*}
Thus, we have completed the proof.
\end{proof}

Recalling that $\Omega^*$ is the solution to problem $(\mathbf{P})$ with the free boundary $\Gamma^*$, i.e., $\Omega^*$
satisfies the following over-determined boundary value problem:
\begin{equation}
    \begin{cases}
    -\Delta u=f\quad {\rm in}\quad \Omega^*, \\
    -\frac{\partial u}{\partial \textbf{n}}=g,\ u=0\quad {\rm on}\quad \Gamma^*, \\
    u=h\quad {\rm on}\quad\Sigma ,
    \end{cases}
\end{equation}
we have the following corollary.
\begin{Corollary}
   The optimal domain $\Omega^*$ satisfies the first order optimality condition:
    \begin{equation}
        dJ(\Omega^*)[dr]=0\quad\forall dr \in X.   
    \end{equation}
\end{Corollary}
\begin{proof}
    It is easy to check that $p(\Omega^*)\equiv 0$ in $\Omega^*$, so we can obtain the result by inserting $p$ into the expression of $dJ(\Omega^*)[V]$.   
    \end{proof}

\subsection{Shape Hessian}
In this subsection we calculate the shape Hessian. To begin with, we first give the change of variables formulae for the unit normal vector and the boundary integrals.
\begin{Lemma}{\rm{(\cite[Chapter 9, pp. 488-489]{delfour2011shapes}, \cite[Chapter 4, pp. 15]{allaire2021shape})}}\label{thm:yinli314}
Let $\Omega \subset \mathbb{R} ^{n}$ be a bounded $C^1$ domain and $T$ be a $C^1$ diffeomorphism of $\mathbb{R} ^{n}$. Then
\begin{equation}
    \textbf{n}_T\circ T=\frac{{}^*(DT)^{-1}\textbf{n}}{\| {}^*(DT)^{-1}\textbf{n} \Vert  },
\end{equation}
where $\textbf{n}$ and $\textbf{n}_T$ are the unit normal vectors of $\partial\Omega$ and $T(\partial\Omega)$ respectively. Moreover, $f\in L^1(T(\partial \Omega))$ if and only if
$f\circ T \in L^1(\partial \Omega) $ and 
\begin{equation}
\int_{T(\partial \Omega)}f d\sigma=
\int_{\partial \Omega} f\circ T \|{}^*(DT)^{-1}\textbf{n} \Vert | \det(DT)\vert d\sigma.\label{con:yinli311eq312}
\end{equation}
\end{Lemma}


Let $r(\widehat{x} )\in X$ be the parametrization of the free boundary of $\Omega$, and let $r_\varepsilon (\widehat{x} )=r(\widehat{x} )+\varepsilon dr(\widehat{x} )$ be the parametrization of the free boundary of $\Omega_{\varepsilon }$, where $dr \in X$. Then the free boundaries before and after transformation are given by
\begin{gather}
\Gamma =\left\{\gamma (\widehat{x} )=r(\widehat{x} )\widehat{x}: \quad\widehat{x} \in \mathbb{S} ^{n-1} \right\} ,\quad 
\Gamma_\varepsilon =\left\{\gamma (\widehat{x} )=r_\varepsilon(\widehat{x} )\widehat{x}: \quad\widehat{x} \in \mathbb{S} ^{n-1} \right\}. \nonumber
\end{gather}

\begin{Theorem}\label{thm:dl316}
Under the above notations and conditions, the following two statements hold:
\begin{itemize}
    \item[(1)] The unit outward normal vector at  $x\in \Gamma$ is given by the following formula:
\begin{equation}
    \textbf{n}(x)=\frac{r(\widehat{x} )\widehat{x}-\nabla _{\mathbb{S} }r(\widehat{x} ) }{\sqrt{r^2(\widehat{x} ) +
    \|\nabla _{\mathbb{S} }r(\widehat{x} ) \Vert ^2}},
\end{equation}
where $\widehat{x}=\frac{x}{\|x \Vert } $ and
$\nabla _{\mathbb{S} }r(\widehat{x} )$ is the tangential gradient of $r$ along
$\mathbb{S}^{n-1}$.

\item[(2)] If $f \in L^1(\Gamma )$, then
\begin{equation}
\int _\Gamma f d\sigma=
\int _{\mathbb{S} ^{n-1}} fr^{n-2}\sqrt{r^2(\widehat{x} ) +
\|\nabla _{\mathbb{S} }r(\widehat{x} ) \Vert ^2} d\widehat{\sigma} .
\end{equation}
\end{itemize}
\end{Theorem}
\begin{proof}
The assertions in this theorem are extensively used in the literature (cf. \cite{eppler2010tracking1,eppler2010tracking2}). Here we include a brief proof for the readability. 

Note that the mapping from the unit sphere $\mathbb{S}^{n-1}$ to $\Gamma$ is given by $T(\widehat{x} )=r(\widehat{x} )\widehat{x} $, and its inverse mapping is $T^{-1}(x)=\frac{x}{\|x \Vert }$. Therefore, we know that $T$ is a diffeomorphism. By the definition of the Jacobian matrix it follows that
\begin{equation}
DT(\widehat{x} )=\widehat{x}{}^{*}\nabla _{\widehat{x} }r(\widehat{x} )+r(\widehat{x} )I,
    \nonumber
\end{equation}
this implies
\begin{equation}
\begin{split}
\det DT(\widehat{x} )
&=r(\widehat{x} )^n\det\left(\frac{1}{r(\widehat{x} )}\widehat{x}{}^{*}\nabla _{\widehat{x} }r(\widehat{x} )+I\right) \\
&=r(\widehat{x} )^n\left(1+\frac{1}{r(\widehat{x} )}{}^{*}\nabla _{\widehat{x} }r(\widehat{x} )\widehat{x}\right) \\
&=r(\widehat{x} )^{n-1}(r(\widehat{x} )+{}^{*}\nabla _{\widehat{x} }r(\widehat{x} )\widehat{x}).
\end{split}
\nonumber
\end{equation}

We use the well-known Sherman-Morrison-Woodbury formula to derive the formula of $DT^{-1}$: 
\begin{equation}
    DT^{-1}(\widehat{x} )
=\frac{1}{r(\widehat{x} )}I-
\frac{\widehat{x}{}^{*}\nabla _{\widehat{x} }r(\widehat{x} ) }{r(\widehat{x} )^2+r(\widehat{x} ){}^{*}\nabla _{\widehat{x} }r(\widehat{x} )\widehat{x}}
    \nonumber
\end{equation}
and
\begin{equation}
\begin{split}
{}^{*}DT^{-1}(\widehat{x} )\widehat{x}
&=\frac{1}{r(\widehat{x} )}\widehat{x} -
\frac{\nabla _{\widehat{x} }r(\widehat{x} ) }{r(\widehat{x} )^2+r(\widehat{x} ){}^{*}\nabla _{\widehat{x} }r(\widehat{x} )\widehat{x}}\\
&=\frac{1}{r(\widehat{x} )}\widehat{x} -
\frac{\nabla_{\mathbb{S} }r(\widehat{x} )+({}^{*}\nabla_{\widehat{x} }r(\widehat{x})\widehat{x})\widehat{x}    }
{r(\widehat{x} )^2+r(\widehat{x} ){}^{*}\nabla _{\widehat{x} }r(\widehat{x} )\widehat{x}}\\
&=\frac{\widehat{x} }{r(\widehat{x} )+{}^{*}\nabla _{\widehat{x} }r(\widehat{x} )\widehat{x}}
-\frac{\nabla_{\mathbb{S} }r(\widehat{x} )}{r(\widehat{x} )(r(\widehat{x} )+{}^{*}\nabla _{\widehat{x} }r(\widehat{x} )\widehat{x})}.
\end{split}
\nonumber
\end{equation}

Since the normal vector at $\widehat{x} $ on the unit sphere is $\widehat{x} $, we have $\left\langle \nabla_{\mathbb{S} }r(\widehat{x} ),\widehat{x} \right\rangle =0$. The first assertion follows from the reduction to a common denominator. Moreover, by a direct calculation, we have\begin{equation}
|\det DT|\|{}^{*}DT^{-1}(\widehat{x} )\widehat{x}  \Vert 
=r(\widehat{x} )^{n-2}\sqrt{r(\widehat{x} )^2+\|\nabla_{\mathbb{S} }r(\widehat{x} ) \Vert ^2}.
\nonumber
\end{equation}
Inserting $|\det DT|\|{}^{*}DT^{-1}(\widehat{x} )\widehat{x}  \Vert $ into (\ref{con:yinli311eq312}), the second assertion follows.
\end{proof}

If $V$ is a velocity field induced by $dr$, then at $x \in \Gamma$ there holds
\begin{equation}\nonumber
    v_n=\left\langle dr(\widehat{x} )\widehat{x},\textbf{n}(x) \right\rangle 
    =  \frac{dr(\widehat{x} )r(\widehat{x} ) }{\sqrt{r^2(\widehat{x} ) +
    \|\nabla _{\mathbb{S} }r(\widehat{x} ) \Vert ^2}}.
\end{equation}
We denote $dJ(\Omega)[V]$ as $dJ(\Omega)[dr]$, by using Theorem \ref{thm:dl316}(2) we obtain
\begin{equation}
        dJ(\Omega)[dr]
        =\int_{\mathbb{S} ^{n-1}} drr^{n-1}
        \left\{ 
            -\frac{1}{2}p^2\mathcal{H} 
            +p\left[f-\mathcal{H} g-\frac{\partial g}{\partial \textbf{n}}\right]
            -\left(g+p\right)\frac{\partial p}{\partial \textbf{n}}
        \right\}  d\widehat{\sigma}.
\end{equation}   

Let 
\begin{equation}\nonumber
    \Omega_{\varepsilon}=\left\{
    x=\rho \widehat{x} \in \mathbb{R} ^{n}:\quad \widehat{x}\in 
    \mathbb{S} ^{n-1},\quad L(\widehat{x} )\leq \rho \leq r(\widehat{x} ) 
    +\varepsilon dr_2(\widehat{x} ) \right\} ,  
\end{equation}
then
\begin{equation}\nonumber
d^2J(\Omega)[dr_1,dr_2]=\frac{d}{d\varepsilon}\bigg|_{\varepsilon=0}\left\{dJ(\Omega_\varepsilon)[dr_1]\right\}  ,
\end{equation}
and we obtain the expression of $ dJ(\Omega_\varepsilon)[dr_1]$
\begin{equation}
    \begin{split}
        dJ(\Omega_\varepsilon)[dr_1]
        &=\int_{\mathbb{S} ^{n-1}} dr_1(\widehat{x} )r^{n-1}_\varepsilon(\widehat{x} )
\left\{
-\frac{1}{2}p_{\varepsilon}(r_\varepsilon(\widehat{x} )\widehat{x})^2\mathcal{H} _{\varepsilon}(r_\varepsilon(\widehat{x} )\widehat{x}) \right.\\
&\left.
+p_{\varepsilon}(r_\varepsilon(\widehat{x} )\widehat{x})\left[f(r_\varepsilon(\widehat{x} )\widehat{x})-\mathcal{H}_{\varepsilon}(r_\varepsilon(\widehat{x} )\widehat{x}) g(r_\varepsilon(\widehat{x} )\widehat{x})-\frac{\partial g(r_\varepsilon(\widehat{x} )\widehat{x})}{\partial \textbf{n}_{\varepsilon}}\right] \right.\\
&\left.
-\left(g(r_\varepsilon(\widehat{x} )\widehat{x})+p_{\varepsilon}(r_\varepsilon(\widehat{x} )\widehat{x})\right)\frac{\partial p_{\varepsilon}(r_\varepsilon(\widehat{x} )\widehat{x})}{\partial \textbf{n}_{\varepsilon}}
\right\} 
        d\widehat{\sigma},
    \end{split}
\end{equation} 
where $r_\varepsilon(\widehat{x} )=r(\widehat{x} )+\varepsilon dr_2(\widehat{x} )$.

By using 
\begin{equation*}
    \left(\frac{\partial w}{\partial \textbf{n}}\right)'(\Gamma;V)
    =\frac{\partial dw[dr_2]}{\partial \textbf{n}}+
    \nabla w\cdot (d\textbf{n}-\nabla^2bV)+\frac{\partial ^2w}{\partial \textbf{n}^2}v_n
\end{equation*}
and the adjoint equation (\ref{adjoint_state}), we can obtain that the shape derivative of $p$  along the velocity field induced by $dr_2$, denoted by $dp[dr_2]$, satisfies the following equation:
\begin{equation}\label{adjoint_Euler}
    \begin{cases}
        -\Delta dp[dr_2]=0 \quad {\rm in}\quad \Omega \\
        dp[dr_2]=-\frac{\partial p}{\partial \textbf{n}}v_n+\frac{\partial dw[dr_2]}{\partial \textbf{n}}+
       \nabla w\cdot (d\textbf{n}-\nabla^2bV)+\frac{\partial ^2w}{\partial \textbf{n}^2}v_n \quad {\rm on}\quad \Gamma,\\
        dp[dr_2]=0 \quad {\rm on}\quad\Sigma .
        \end{cases}
\end{equation}
Now we are ready to compute the shape Hessian. Noticing that
\begin{gather*}
\lim_{\varepsilon \to 0}  \frac{r_\varepsilon(\widehat{x} )-r(\widehat{x} )}{\varepsilon}=dr_2(\widehat{x} ), \nonumber \\
\lim_{\varepsilon \to 0}  \frac{p_\varepsilon(r_\varepsilon(\widehat{x})\widehat{x} )-p(r(\widehat{x})\widehat{x} )}{\varepsilon}=dp[dr_2]+dr_2(\widehat{x} ) \left\langle \nabla p,\widehat{x} \right\rangle , \\
\lim_{\varepsilon \to 0}  \frac{f(r_\varepsilon(\widehat{x})\widehat{x} )-f(r(\widehat{x})\widehat{x} )}{\varepsilon}=dr_2(\widehat{x} ) \left\langle \nabla f,\widehat{x} \right\rangle , \nonumber \\
\lim_{\varepsilon \to 0}  \frac{g(r_\varepsilon(\widehat{x})\widehat{x} )-g(r(\widehat{x})\widehat{x} )}{\varepsilon}=dr_2(\widehat{x} ) \left\langle \nabla g,\widehat{x} \right\rangle , \nonumber 
\end{gather*}
and 
\begin{equation*}
    \begin{split}
       \lim_{\varepsilon \to 0} \frac{(\nabla g)(r_\varepsilon(\widehat{x})\widehat{x} )-(\nabla g)(r(\widehat{x})\widehat{x} )}{\varepsilon}
    =\overset{\circ}{(\nabla g)}(\Omega;V)|_{\Gamma^*}
    =(\nabla g)'(\Omega;V)+\nabla^2gV
    =dr_2(\widehat{x})\nabla^2g\widehat{x},\\
   \lim_{\varepsilon \to 0} \frac{(\nabla p_\varepsilon)(r_\varepsilon(\widehat{x})\widehat{x} )-(\nabla p)(r(\widehat{x})\widehat{x} )}{\varepsilon}
=\overset{\circ}{(\nabla p)}(\Omega;V)|_{\Gamma}
=(\nabla p)'(\Omega;V)+\nabla^2pV
=\nabla dp[dr_2]+dr_2\nabla^2p\widehat{x},
\end{split}    
\end{equation*}
we differentiate $dJ(\Omega_\varepsilon)[dr_1]$ with respect to $\varepsilon$ to obtain 
\begin{equation}
    \begin{split}
    d^2J(\Omega)[dr_1,dr_2]
    &=\int_{\mathbb{S} ^{n-1}} (n-1)dr_1r^{n-2}
        \left\{
            -\frac{1}{2}p^2\mathcal{H} 
            +p\left[f-\mathcal{H} g-\frac{\partial g}{\partial \textbf{n}}\right]
            -\left(g+p\right)\frac{\partial p}{\partial \textbf{n}}
        \right\}  d\widehat{\sigma}\\
&+\int_{\mathbb{S} ^{n-1}} dr_1r^{n-1}
\left\{-p\mathcal{H}\left(
dp[dr_2]+dr_2\left\langle \nabla p,\widehat{x}\right\rangle
\right) -\frac{1}{2}p^2\frac{d\mathcal{H}}{d(dr_2)}
 \right\} 
d\widehat{\sigma}\\
&+\int_{\mathbb{S} ^{n-1}} dr_1r^{n-1}
\left\{
\left(dp[dr_2]+dr_2\left\langle \nabla p,\widehat{x}\right\rangle \right) 
\left[f-\mathcal{H} g-\frac{\partial g}{\partial \textbf{n}}\right]  
\right\} 
d\widehat{\sigma}\\
&+\int_{\mathbb{S} ^{n-1}} dr_1r^{n-1}
\Big\{p \Big[
dr_2\left\langle \nabla f,\widehat{x}\right\rangle 
-\frac{d\mathcal{H}}{d(dr_2)}g-\mathcal{H}dr_2
\left\langle \nabla g,\widehat{x}\right\rangle \\
&-\left(dr_2\left\langle \nabla^2g \widehat{x},\textbf{n}\right\rangle
+\left\langle \nabla g,d\textbf{n}[dr_2]\right\rangle 
\right) 
\Big]
\Big\} 
d\widehat{\sigma}\\
&-\int_{\mathbb{S} ^{n-1}} dr_1r^{n-1}
\left\{
    \frac{\partial p}{\partial \textbf{n}}
    \left(dr_2\left\langle \nabla g,\widehat{x}\right\rangle 
+dp[dr_2]+dr_2\left\langle \nabla p,\widehat{x}\right\rangle \right) \right.\\
&\left.
+\left(g+p\right) \left(
    \frac{\partial dp[dr_2]}{\partial \textbf{n}}
    +dr_2\left\langle \nabla^2p\widehat{x},\textbf{n}\right\rangle 
+\left\langle \nabla p,d\textbf{n}[dr_2]\right\rangle 
    \right) 
\right\} 
d\widehat{\sigma},
\end{split}
\end{equation}
where $\frac{d\mathcal{H}}{d(dr_2)}$ represents the material derivative of the additive curvature $\mathcal{H}$.

\begin{Corollary}
    The following identity holds:
    \begin{equation}
        \begin{split}
    d^2J(\Omega^*)[dr_1,dr_2]
    &=\int_{\mathbb{S} ^{n-1}}dr_1r^{*(n-1)}
    \left\{
        dp[dr_2]\left[f-\mathcal{H}  g-\frac{\partial g}{\partial \textbf{n}}\right]   
    -g\frac{\partial dp[dr_2]}{\partial \textbf{n}}
    \right\} d\widehat{\sigma}.
        \end{split}
    \end{equation}
\end{Corollary}
\begin{proof}
    It is easy to know that $p(\Omega^*)\equiv 0$ in $\Omega^*$, so we can get the identity by inserting $p$ into the expression of $d^2J(\Omega^*)[dr_1,dr_2]$.
\end{proof}

Since the expression of $d^2J(\Omega)[dr_1,dr_2]$
contains $\frac{d\mathcal{H}}{d(dr_2)}$,
it is obvious that $dJ(\Omega)[\cdot,\cdot]$
is a continuous bilinear functional on $H^1(\Gamma) \times H^1(\Gamma)$ (cf. \cite{eppler2010tracking1,eppler2010tracking2,eppler2012kohn,eppler2000optimal,eppler2007convergence}), and the second-order Taylor remainder is 
\begin{equation*}
    \left\lvert R_2(J(\Omega),dr) \right\rvert 
    =o(\left\lVert dr\right\rVert _{C^{3,\alpha}(\Gamma)})
    \left\lVert dr\right\rVert _{H^1(\Gamma)}^2.
\end{equation*}
Therefore, by using Theorem \ref{thm:dl24},
$\Omega^*$ is a strictly local minimizer if
\begin{equation*}
    d^2J(\Omega^*)[dr,dr]\gtrsim \left\lVert dr\right\rVert _{H^1(\Gamma^*)}^2.
\end{equation*}
Moreover, the coercivity of the shape Hessian also ensures the existence of local optimal solutions for the nonlinear Ritz-Galerkin approximation method and its convergence according to Theorem \ref{thm:DL214}.


To prove the coervivity of the shape Hessian, we first search for an operator representation following the ideas of \cite{eppler2010tracking1,eppler2010tracking2}. 
\begin{Lemma}[\cite{eppler2010tracking1,eppler2010tracking2}]
Let $\textbf{n}$ be the unit normal vector of $\Gamma^*$, then the operator
\begin{equation}
M(dr)\triangleq g\left\langle dr\widehat{x},\textbf{n} \right\rangle 
=\frac{gr^*}{\sqrt{r^{*2}+\|\nabla_{\mathbb{S} }r^{*} \Vert^2 }}dr
\nonumber
\end{equation}
is a continuous and bijective mapping from $H^s(\Gamma^*)$ to $H^s(\Gamma^*)$ for all $s\in [0,1]$.
\end{Lemma}

\begin{Definition}[\cite{eppler2010tracking1,eppler2010tracking2}]
Let $g(x)>0$ for all $x\in \Gamma^*$. We introduce the operator $\mathcal{A}:H^s(\Gamma^*)\to H^s(\Gamma^*)$, $s\in [0,1]$ defined by
    \begin{equation}
(\mathcal{A}u)(x) \triangleq A(x)u(x)
= \left\{\mathcal{H}(x)+\left[\frac{\partial g(x)}{\partial \textbf{n}}
-f(x)\right] \bigg/g(x)\right\} u(x).
\nonumber
    \end{equation} 
\end{Definition}

\begin{Definition}We define the 
{\rm {Dirichlet-to-Neumann}} mapping 
 $\Lambda: H^{1/2}(\Gamma^*) \to H^{-1/2}(\Gamma^*)$ as $\Lambda(v)\triangleq\frac{\partial w}{\partial \textbf{n}} \big|_{\Gamma^*}$,
where $w \in H^{1}(\Omega^*)$ satisfies the following equation:
\begin{equation}
    \begin{cases}
        -\Delta w=0\quad in\quad \Omega^* ,\\
        w=v\quad on\quad \Gamma^* ,\\
        w=0\quad on\quad\Sigma .
        \end{cases} 
    \nonumber
\end{equation}
\end{Definition} 
\begin{Lemma}[\cite{eppler2010tracking2}]
The mapping $\Lambda$ is $H^{1/2}(\Gamma^*)$-$coercive$, and its inverse is
the {\rm {Neumann-to-Dirichlet}} mapping $\Upsilon =\Lambda^{-1}$,
i.e., $\Upsilon: H^{-1/2}(\Gamma^*) \to H^{1/2}(\Gamma^*)$ defined by $\Upsilon (v):=w|_{\Gamma^*}$,
where $w \in H^{1}(\Omega^*)$ satisfies the following equation:
\begin{equation}
    \begin{cases}
        -\Delta w=0\quad {\rm in}\quad \Omega^* ,\\
        \frac{\partial w}{\partial \textbf{n}}=v\quad {\rm on}\quad \Gamma^* ,\\
        w=0\quad {\rm on}\quad\Sigma .
        \end{cases} 
    \nonumber
\end{equation}
\end{Lemma}

Now we are ready to rewrite $d^2J(\Omega^*)[dr_1,dr_2]$ into an operator form. 
\begin{Theorem}
The following identity holds:
\begin{equation}
d^2J(\Omega^*)[dr_1,dr_2]=
\left\langle(\mathcal{A}+\Lambda)(M(dr_1)),(\mathcal{A}+\Lambda)(M(dr_2)) \right\rangle _{L^2(\Gamma^*)}\quad \forall dr_1,dr_2 \in X.
\end{equation}
\end{Theorem}
\begin{proof}
Let $dp[dr_2]$ be the solution to (\ref{adjoint_Euler}) with $\Omega$ and $\Gamma$ replaced by $\Omega^*$ and $\Gamma^*$. 
Recalling that on $\Gamma^*$ there holds
\begin{equation}
\begin{split}
    \frac{\partial du[dr_2]}{\partial \textbf{n}}=
    \left\langle dr_2\widehat{x},\textbf{n} \right\rangle \left[f-\mathcal{H}  g-\frac{\partial g}{\partial \textbf{n}}\right]
    =-\mathcal{A}(M(dr_2)),
\end{split}
\nonumber
\end{equation}
we have
\begin{equation}
du[dr_2]=-\Upsilon (\mathcal{A}(M(dr_2))).
    \nonumber
\end{equation}
By using Lemma \ref{thm:yinli11} and the fact that $w|_{\Gamma^*}=0$ on $\Gamma ^*$, we have
\begin{equation*}
\begin{split}
    dw[dr_2]=du[dr_2]-g    \left\langle dr_2\widehat{x},\textbf{n}\right\rangle
    =-\Upsilon (\mathcal{A}(M(dr_2)))-M(dr_2),
\end{split}
\end{equation*}
this implies
\begin{equation*}
    \frac{\partial dw[dr_2]}{\partial \textbf{n}}
=\Lambda (-\Upsilon (\mathcal{A}(M(dr_2)))-M(dr_2))
=-\mathcal{A} (M(dr_2))-\Lambda (M(dr_2)).
\end{equation*}
Moreover, we can obtain 
\begin{equation*}
    dp[dr_2]=\frac{\partial dw[dr_2]}{\partial \textbf{n}}
    =-\mathcal{A} (M(dr_2))-\Lambda (M(dr_2))
\end{equation*}
and 
\begin{equation*}
    \frac{\partial dp[dr_2]}{\partial \textbf{n}} 
    =-\Lambda (\mathcal{A} (M(dr_2))+\Lambda (M(dr_2))).
\end{equation*}
Let $v$ satisfy
\begin{equation*}
    \begin{cases}
        -\Delta v=0 \quad {\rm in}\quad \Omega^*, \\
        v=M(dr_1) \quad {\rm on}\quad \Gamma^* ,\\
        v=0 \quad {\rm on}\quad\Sigma ,
        \end{cases}
\end{equation*}
by using Green's formula we have
\begin{equation*}
    \begin{split}
        \int_{\Gamma^*}M(dr_1)\frac{\partial dp[dr_2]}{\partial \textbf{n}} d\sigma
        &= \int_{\partial \Omega^*}\frac{\partial v}{\partial \textbf{n}}dp[dr_2] d\sigma
=\int_{\Gamma^*}\Lambda(M(dr_1))dp[dr_2] d\sigma.
    \end{split}
\end{equation*}
Combing the above computations, we have
\begin{equation*}
    \begin{split}
d^2J(\Omega^*)[dr_1,dr_2]
&=\int_{\mathbb{S} ^{n-1}}dr_1r^{*(n-1)}
\left\{
    dp[dr_2]\left[f-\mathcal{H}  g-\frac{\partial g}{\partial \textbf{n}}\right]   
-g\frac{\partial dp[dr_2]}{\partial \textbf{n}}
\right\} d\widehat{\sigma}\\
&= \int_{\Gamma^*}\left\langle dr_1 \widehat{x},\textbf{n}(x) \right\rangle 
dp[dr_2]\left[f-\mathcal{H}  g-\frac{\partial g}{\partial \textbf{n}}\right]  d\sigma
-\int_{\Gamma^*}M(dr_1)\frac{\partial dp[dr_2]}{\partial \textbf{n}}d\sigma\\
&= \int_{\Gamma^*}g\left\langle dr_1 \widehat{x},\textbf{n}(x) \right\rangle 
dp[dr_2]\left[f-\mathcal{H}  g-\frac{\partial g}{\partial \textbf{n}}\right]\bigg/g  d\sigma
-\int_{\Gamma^*}M(dr_1)\frac{\partial dp[dr_2]}{\partial \textbf{n}}d\sigma\\    
&=-\int_{\Gamma^*} \mathcal{A} (M(dr_1))dp[dr_2]d\sigma
-\int_{\Gamma^*}\Lambda(M(dr_1))dp[dr_2] d{\sigma}\\
&=\left\langle \mathcal{A} (M(dr_1)),\mathcal{A} (M(dr_2))+\Lambda (M(dr_2))\right\rangle _{L^2(\Gamma^*)}\\
&\quad+\left\langle \Lambda (M(dr_1)),\mathcal{A} (M(dr_2))+\Lambda (M(dr_2))\right\rangle _{L^2(\Gamma^*)}\\
&=\left\langle(\mathcal{A}+\Lambda)(M(dr_1)),(\mathcal{A}+\Lambda)(M(dr_2)) \right\rangle _{L^2(\Gamma^*)}.
\end{split}
\end{equation*}
This finishes the proof.
\end{proof}

In the following we will prove the main result of this section, i.e., the coercivity of the shape Hessian. 
\begin{Theorem}\label{thm:dl321}
If
\begin{equation}
    A(x)=\mathcal{H} (x)+\left[\frac{\partial g(x)}{\partial \textbf{n}}
    -f(x)\right] \bigg/g(x)\geqslant 0\quad\forall x \in \Gamma^*
    \nonumber
\end{equation}
and $A \in L^{\infty}(\Gamma^*)$, then
\begin{equation}
d^2J(\Omega^*)[dr,dr]\gtrsim \|dr \Vert ^2_{H^{1}(\Gamma^*)}.
\end{equation}
\end{Theorem}
\begin{proof}
On the one hand, by using the definitions of $\Lambda$ and $\mathcal{A}$,
we have
\begin{equation*}
(\Lambda+\mathcal{A})(M(dr))=\frac{\partial w^*}{\partial \textbf{n}}\bigg|_{\Gamma ^*}
+Aw^*,
\end{equation*}
where $w^*$ satisfies the following equation:
\begin{equation*}
    \begin{cases}
        -\Delta w^*=0\quad {\rm in}\quad \Omega^*, \\
        w^*=M(dr)\quad {\rm on}\quad \Gamma^* ,\\
        w^*=0\quad {\rm on}\quad\Sigma .
        \end{cases} 
\end{equation*}
By recalling the trace theorem, we have
\begin{equation*}
\left\lVert w^* \right\rVert _{H^{3/2}(\Omega^*)}
\gtrsim \left\lVert M(dr) \right\rVert  _{H^{1}(\Gamma^*)}
\thicksim  \left\lVert dr \right\rVert  _{H^{1}(\Gamma^*)}.
\end{equation*}
On the other hand, let $v\triangleq \frac{\partial w^*}{\partial \textbf{n}}+Aw^*$,
then
\begin{equation*}
    \begin{cases}
        -\Delta w^*=0\quad {\rm in}\quad \Omega^*, \\
        \frac{\partial w^*}{\partial \textbf{n}}+Aw^*=v\quad {\rm on}\quad \Gamma^*, \\
        w^*=0\quad {\rm on}\quad\Sigma.
        \end{cases} 
\end{equation*}
Using the well-posedness of second elliptic equations with Robin boundary conditions, we obtain
\begin{equation*}
\left\lVert w^* \right\rVert _{H^{3/2}(\Omega^*)}
\lesssim  
\left\lVert v \right\rVert  _{L^{2}(\Gamma^*)}
=\left\lVert (\Lambda+\mathcal{A})(M(dr)) \right\rVert  _{L^{2}(\Gamma^*)}.
\end{equation*}

Collecting the above two results, we arrive at
\begin{equation}
    \left\lVert (\mathcal{A}+\Lambda )(M(dr))\right\rVert _{L^{2}(\Gamma^*)}
    \gtrsim  \left\lVert dr \right\rVert _{H^{1}(\Gamma^*)}.
        \nonumber
\end{equation} 
Therefore, we have
\begin{equation}
d^2J(\Omega^*)[dr,dr]\gtrsim \|dr \Vert ^2_{H^{1}(\Gamma^*)},
\nonumber
\end{equation}
this completes the proof.
\end{proof}

\begin{Remark}[\cite{eppler2010tracking1}]\label{remark:rmk322}
If the domain $\{x=\rho\widehat{x}\in \mathbb{R} ^n:\widehat{x}\in \mathbb{S}^{n-1},0\leqslant \rho \leqslant r^*(\widehat{x}) \}$ is convex, $g$ is a positive constant function and $f\equiv0$, then 
\begin{equation}
d^2J(\Omega^*)[dr,dr]\gtrsim \|dr \Vert ^2_{H^{1}(\Gamma^{*})}.
\nonumber
\end{equation}
\end{Remark}
As a result, $J$ is a well-posed objective functional which tracks Dirichlet data at the free boundary.

\section{Motivations and discussions}

In this section we explain in detail the motivation to propose the objective functional $\frac{1}{2}\int_{\Gamma}\left(\frac{\partial w}{\partial \textbf{n}}\right)^2 d\sigma$. 

As was shown in \cite{eppler2010tracking2}, tracking Dirichlet data in $L^2(\Gamma)$ is ill-posed. An intuitive explanation is that the $L^2(\Gamma)$ regularity for the Dirichlet data is too weak. Because the data $f, g, h$ and the domain $\Omega$ are sufficiently smooth, we can expect a higher regularity of the solution $u$ and thus the Dirichlet data. 

Motivated by this observation, our starting point is to use a stronger norm on the free boundary to track Dirichlet data. Our first choice is to enhance the objective functional $J_2(\Omega)=\frac{1}{2}\|u\| _{L^2(\Gamma)}^2 $ to $\widetilde{J}(\Omega)=\frac{1}{2}\|u\| _{H^{1/2}(\Gamma)}^2$ to track Dirichlet data. We remark that $\| u\| _{H^{1/2}(\Gamma)}$ is the correct energy space for Dirichlet data in view of the standard variational solution to second order elliptic equations. There are several different equivalent definitions for the norm $\| \cdot \|_{H^{1/2}(\Gamma)}$, here we choose the one by using the Dirichlet-to-Neumann map for performing the shape calculus conveniently. Let
\begin{equation}
        \begin{cases}
        -\Delta u=f \quad {\rm in}\quad \Omega, \\
        -\frac{\partial u}{\partial \textbf{n}}=g \quad {\rm on}\quad \Gamma ,\\
        u=h \quad {\rm on}\quad\Sigma 
        \end{cases}\quad
        \text{and}\quad
        \begin{cases}
            -\Delta w=0 \quad {\rm in}\quad \Omega ,\\
            w=u \quad {\rm on}\quad\Gamma, \\
            w=0 \quad {\rm on}\quad \Sigma.
        \end{cases}
        \nonumber
\end{equation}
It is not difficult to prove that (cf. \cite{Gong})
\begin{equation}
\left\lVert u \right\rVert _{H^{1/2}(\Gamma)}^2 \approx \langle u,\mathcal{D}u\rangle_{H^{1/2}(\Gamma),H^{-1/2}(\Gamma)}=\langle u,\frac{\partial w}{\partial \textbf{n}} \rangle_{H^{1/2}(\Gamma),H^{-1/2}(\Gamma)},    
\end{equation}
where $\mathcal{D}:H^{1/2}(\Gamma)\rightarrow H^{-1/2}(\Gamma)$ is the standard Dirichlet-to-Neumann map such that $\mathcal{D}u:=\frac{\partial w}{\partial \textbf{n}}$. As a result, we can directly set $\widetilde{J}(\Omega)=\frac{1}{2}\langle u,\frac{\partial w}{\partial \textbf{n}} \rangle_{H^{1/2}(\Gamma),H^{-1/2}(\Gamma)}$.

We can obtain the Euler derivative of $\widetilde{J}(\Omega)$ as follows: 
\begin{equation}
    d\widetilde{J}(\Omega)[V]
   =\frac{1}{2}\int_{\Gamma} v_{n}
    \left\{ -\nabla w\cdot \nabla w
        -2g\frac{\partial w}{\partial \textbf{n}}
    +2w\left[f-\mathcal{H} g-\frac{\partial g}{\partial \textbf{n}}\right] 
    \right\} d\sigma.
\end{equation}
The shape Hessian at the optimum $\Omega^*$ is given by:
\begin{equation}
        \begin{split}
    d^2\widetilde{J}(\Omega^*)[dr_1,dr_2]
    &=\int_{\mathbb{S} ^{n-1}}dr_1r^{*(n-1)}
    \left\{
        dw[dr_2]\left[f-\mathcal{H}  g-\frac{\partial g}{\partial \textbf{n}}\right]   
    -g\frac{\partial dw[dr_2]}{\partial \textbf{n}}
    \right\} d\widehat{\sigma}.
        \end{split}
    \end{equation}
Analog to the above section, we have proven that 
\begin{equation}
    d^2\widetilde{J}(\Omega^*)[dr_1,dr_2]=
\left\langle\Upsilon ((\mathcal{A}+\Lambda)(M(dr_1))),
(\mathcal{A}+\Lambda)(M(dr_2)) \right\rangle _{L^2(\Gamma^*)}
\quad \forall dr_1,dr_2 \in X
\end{equation}
and the shape Hessian is a bilinear continuous functional on $H^1(\Gamma)\times H^1(\Gamma)$ in the neighborhood $B_\delta ^X(r^*)$ of $\Omega^*$, where the definitions of $\Upsilon, \mathcal{A}, M$ and $\Lambda$ are the same with that in the above section. However, we only proved that
\begin{equation*}
d^2\widetilde{J}(\Omega^*)[dr,dr]\geq c_E\left\lVert dr\right\rVert _{H^{1/2}(\Gamma^*)}^2 \quad \forall dr \in X,
\end{equation*}
so we can not determine whether $\widetilde{J}$ is well-posed. 

Recalling that tracking Neumann data in $L^2(\Gamma)$ is well-posed (cf. \cite{eppler2010tracking1}), and the energy space for Neumann data allowing for a standard variational solution to second order elliptic equations is $H^{-1/2}(\Gamma)$. Let $u$ be the solution to problem (\ref{con:j1}), if we can show that the shape functional $\breve{J}(\Omega)={1\over 2}\|\frac{\partial u}{\partial \textbf{n}}+g\|_{H^{-1/2}(\Gamma)}^2$ is also not well-posed, we can say that tracking boundary conditions in energy spaces are not enough to ensure the well-posedness. To verify this conjecture, we study the well-posedness of the objective functional $\breve{J}(\Omega)$. Since the standard definition of the norm $\|\cdot\|_{H^{-1/2}(\Gamma)}$ by using duality is not convenient for performing shape calculus, we use an equivalent definition through the Neumann-to-Dirichlet map. Assume that $u$ is the solution to problem (\ref{con:j1}), let $z\in H^1(\Omega)$ satisfy
\begin{equation*}
    \begin{cases}
    -\Delta z=0 \quad {\rm in}\quad \Omega, \\
    \frac{\partial z}{\partial n}=\frac{\partial u}{\partial \textbf{n}}+g \quad {\rm on}\quad \Gamma, \\
    z=0 \quad {\rm on}\quad\Sigma,
    \end{cases}
\end{equation*}
it is not difficult to prove that (cf. \cite{Apel})
\begin{equation}
\left\lVert \frac{\partial u}{\partial \textbf{n}}+g \right\rVert _{H^{-1/2}(\Gamma)}^2 \approx \Big\langle \frac{\partial u}{\partial \textbf{n}}+g,\mathcal{N}(\frac{\partial u}{\partial \textbf{n}}+g)\Big\rangle_{H^{-1/2}(\Gamma),H^{1/2}(\Gamma)}=\langle \frac{\partial u}{\partial \textbf{n}}+g,z \rangle_{H^{-1/2}(\Gamma),H^{1/2}(\Gamma)},    
\end{equation}
where $\mathcal{N}:H^{-1/2}(\Gamma)\rightarrow H^{1/2}(\Gamma)$ is the standard Neumann-to-Dirichlet map such that $\mathcal{N}(\frac{\partial u}{\partial \textbf{n}}+g):=z|_{\Gamma}$. As a result, we can directly set $\breve{J}(\Omega)=\frac{1}{2}\langle z,\frac{\partial u}{\partial \textbf{n}}+g\rangle_{H^{1/2}(\Gamma),H^{-1/2}(\Gamma)}$.

We can obtain the Euler derivative of $\breve{J}(\Omega)$ as follows: 
\begin{equation}
   d\breve{J}(\Omega)[V]=
   \int_{\Gamma}v_n\left\{-\frac{\partial u}{\partial \textbf{n}}
   \frac{\partial z}{\partial \textbf{n}}+z\mathcal{H}\frac{\partial z}{\partial \textbf{n}}+z\left[ \frac{\partial g}{\partial \textbf{n}}
   +\frac{\partial ^2u}{\partial \textbf{n}^2}
   \right]-\frac{1}{2}\nabla_{\Gamma}z\cdot\nabla_{\Gamma}z
   \right\}d\sigma.
\end{equation}
The shape Hessian at the optimum $\Omega^*$ is given by:
 \begin{equation}
        \begin{split}
    d^2\breve{J}(\Omega^*)[dr_1,dr_2]
    &=\int_{\mathbb{S} ^{n-1}}dr_1r^{*(n-1)}
    \left\{
        dz[dr_2]\left[f-\mathcal{H}  g-\frac{\partial g}{\partial \textbf{n}}\right]   
    -g\frac{\partial dz[dr_2]}{\partial \textbf{n}}
    \right\} d\widehat{\sigma}.
        \end{split}
    \end{equation}
Moreover,
\begin{equation*}
    d^2\breve{J}(\Omega^*)[dr_1,dr_2]=
\left\langle\Upsilon ((\mathcal{A}+\Lambda)(M(dr_1))),(\mathcal{A}+\Lambda)(M(dr_2)) \right\rangle _{L^2(\Gamma^*)} \quad \forall dr_1,dr_2 \in X.
    \nonumber
\end{equation*}
The situation is completely the same with that of $\widetilde{J}$, so we cannot determine whether $\breve{J}$ is well-posed.

Combining the above results for tracking Dirichlet or Neumann data in energy spaces we see that the energy space is not sufficient for the well-posedness. This observation also explains why the problem (\ref{J4_obj})-(\ref{J4_state}) is not well-posed. However, the well-posedness of tracking Neumann data in $L^2(\Gamma)$ encourages us to use a stronger norm that $H^{1/2}(\Gamma)$, this motivates our second choice of the objective functional $\widehat{J}(\Omega)={1\over 2}\left\lVert u \right\rVert _{H^1(\Gamma)}^2$. That is, we consider the following shape optimization problem 
\begin{equation}
    \inf_{\Omega } \widehat{J}(\Omega )=\frac{1}{2}\int_{\Gamma}u^2+\nabla_{\Gamma}u \cdot \nabla_{\Gamma}u d\sigma,
\end{equation}
subject to the Neumann problem (\ref{con:j2}). This choice of objective functional was also suggested in \cite{leugering2012constrained} but no analysis is available. 
However, we encountered great difficulties when dealing with the shape Hessian, because we shall compute the shape derivative of $\Delta_\Gamma u$ on the boundary. In addition, there are also difficulties in rewriting the shape Hessian into an operator form. Nevertheless, by observing the algebraic structure of the choices of $\frac{1}{2}\int_{\Gamma}u^2d\sigma$ and $\frac{1}{2}\langle u,\frac{\partial w}{\partial \textbf{n}} \rangle_{H^{1/2}(\Gamma),H^{-1/2}(\Gamma)}$, we choose $\frac{1}{2}\int_{\Gamma}\left(\frac{\partial w}{\partial \textbf{n}}\right)^2d\sigma$ and expect it is an equivalent norm with ${1\over 2}\left\lVert u \right\rVert _{H^1(\Gamma)}^2$. Indeed, we can show the equivalence in the following lemma.

\begin{Lemma}
Let $(u,w)$ be the solutions of (\ref{con:equ1d6}). We have the equivalence 
\begin{equation*}
  \left\lVert u \right\rVert _{H^1(\Gamma)}^2\approx  \int_{\Gamma}\left(\frac{\partial w}{\partial \textbf{n}}\right)^2d\sigma.
\end{equation*}
\end{Lemma}
\begin{proof}
  On the one hand, by recalling the trace theorem and the well-posedness of second order elliptic equations with Dirichlet boundary conditions, we have
   \begin{equation}
       \int_{\Gamma}\left(\frac{\partial w}{\partial \textbf{n}}\right)^2d\sigma
       =\left\lVert\frac{\partial w}{\partial \textbf{n}}
        \right\rVert_{L^2(\Gamma)}^2
       \lesssim 
       \left\lVert w
        \right\rVert_{H^{3/2}(\Omega)}^2
        \lesssim 
       \left\lVert u
        \right\rVert_{H^{1}(\Gamma)}^2.
   \end{equation}
   On the other hand, let $\widehat{g}:=\frac{\partial w}{\partial n}$ on $\Gamma$, then
   \begin{equation*}
    \begin{cases}
        -\Delta w=0\quad {\rm in}\quad \Omega, \\
        \frac{\partial w}{\partial \textbf{n}}=\widehat{g}\quad {\rm on}\quad \Gamma, \\
        w=0\quad {\rm on}\quad\Sigma.
        \end{cases} 
\end{equation*}
   Therefore, by the trace theorem and the well-posedness of second order elliptic equations with Neumann boundary conditions, we have
   \begin{equation*}
       \left\lVert u
        \right\rVert_{H^{1}(\Gamma)}
        =\left\lVert w
        \right\rVert_{H^{1}(\Gamma)}
        \lesssim
         \left\lVert w
        \right\rVert_{H^{3/2}(\Omega)}
        \lesssim
        \left\lVert \widehat{g}
        \right\rVert_{L^2(\Gamma)}
        =\left\lVert\frac{\partial w}{\partial \textbf{n}}\right\rVert_{L^2(\Gamma)}.
   \end{equation*}
 Collecting the above two results, we arrive at the conclusion.
\end{proof}

Now, we summarize the known results for tracking different boundary conditions of the Bernoulli free boundary problem in Table \ref{tab:biaosum}. We see that tracking Dirichlet and Neumann data are both well-posed by choosing appropriate objective functionals. 
\begin{table}[h]
    \centering
    \caption{The well-posedness of different objective functionals.}
    \label{tab:biaosum}
\begin{tabular}{|c|c|c|c|c|}
    \hline
    objective&continuity space &coercivity space &tracking data type&well-posedness\\
    \hline
    $J_1$ (\cite{eppler2010tracking1})&$H^1(\Gamma)$& $H^1(\Gamma)$&Neumann & well-posed \\
    \hline
    $J_2$ (\cite{eppler2010tracking2})&$H^1(\Gamma)$& $L^2(\Gamma)$&Dirichlet&algebraically ill-posed\\
    \hline
    $J_3$ (\cite{eppler2006efficient})&$H^{1/2}(\Gamma)$& $H^{1/2}(\Gamma)$&Neumann&well-posed\\
    \hline
    $J_4$ (\cite{eppler2012kohn})&$H^{1}(\Gamma)$& $H^{1/2}(\Gamma)$&Neumann+Dirichlet&algebraically ill-posed \\
    \hline
    $\breve{J}$ (this paper)&$H^{1}(\Gamma)$& $H^{1/2}(\Gamma)$& Neumann&algebraically ill-posed \\
    \hline
    $\widetilde{J}$ (this paper) &$H^{1}(\Gamma)$& $H^{1/2}(\Gamma)$& Dirichlet&algebraically ill-posed \\
    \hline
     $J$  (this paper)&$H^{1}(\Gamma)$& $H^{1}(\Gamma)$& Dirichlet&well-posed \\
    \hline
\end{tabular}
\end{table}

\section{Numerical experiments}
In this section we use the nonlinear Ritz-Galerkin approximation method (cf. \cite{eppler2010tracking1}) to solve the shape optimization problem (\ref{con:jnew})-(\ref{con:equ1d6}). 

\subsection{Nonlinear Ritz-Galerkin approximation method}
For simplicity, we assume that the domain $\Omega$ is a star-shaped domain. According to the previous notation, its free boundary is parameterized by
\begin{equation*}
    \Gamma=\left\{x=dr(\widehat{x})\widehat{x}: \quad\widehat{x}\in\mathbb{S}^{n-1} \right\} ,    
\end{equation*}
where $dr\in X $. 
Consequently, we  can transform the shape optimization problem (\ref{con:jnew})-(\ref{con:equ1d6}) into 
\begin{equation*}
    \inf_{r\in X}J(r).
\end{equation*}
The so-called Ritz-Galerkin method is to select a finite-dimensional subspace $V_N\subset X$ and to solve the finite-dimensional optimization problem:
\begin{equation*}
    (\mathbf{P}_N)\quad \inf_{r_N\in V_N}J(r_N).
\end{equation*}
Herein, the dimension of the space  $V_N$ depends linearly on $N$. As can be seen from the following text, the space $V_N$ we selected has a dimension of $2N+1$.

Since the global optimization problem is difficult to deal with and the well-posedness of the objective functional $J$ is also local, we consider the following local optimization problem:
\begin{equation*}
    (\mathbf{P}^{\delta }) \quad \inf_{r}J(r), \quad r\in \overline{B_{\delta}^X(r^*)},
\end{equation*}
where $r^*$ is the parametrization of the free boundary of $\Omega^*$. Correspondingly, the Ritz-Galerkin discretization of this local problem is
\begin{equation}
    (\mathbf{P}^{\delta }_N)\quad \inf_{r}J(r),\quad r\in V_N\cap \overline{B_{\delta}^X(r^*)}.
\end{equation}


Let $V$ be the velocity field induced by $dr$ and $t$ be the transformation step, we can define  the transform of the free boundary  from 
\begin{equation*}
    \Gamma=\left\{r(\widehat{x})\widehat{x}:\quad \widehat{x} \in \mathbb{S} ^{n-1} \right\} 
\end{equation*}
to
\begin{equation*}
    \Gamma_t=\left\{r(\widehat{x})\widehat{x}+tdr(\widehat{x})\widehat{x}:\quad \widehat{x} \in \mathbb{S} ^{n-1} \right\}. 
\end{equation*}
For simplicity, we only consider the case $n=2$, the three-dimensional case can be found in \cite{harbrecht2008newton,eppler2010tracking1}.
Then the free boundary can be equivalently parameterized as
\begin{gather*}
    \Gamma=\left\{r(\theta )\binom{\cos\theta}{\sin\theta} :\quad\theta \in [0,2\pi] \right\} ,\quad
    \Gamma_t=\left\{r(\theta)\binom{\cos\theta}{\sin\theta}+tdr(\theta)\binom{\cos\theta}{\sin\theta}:\quad\theta \in [0,2\pi] \right\}.
\end{gather*}

Since $r(\theta)$ is a periodic function, it has a Fourier series expansion and we choose the first $N$ terms truncation. That is 
\begin{equation*}
    r(\theta)\approx a_0+\sum_{i= 1}^{N}a_i  \cos(i\theta)
    +\sum_{i= 1}^{N}b_i  \sin(i\theta).
\end{equation*}
Let 
\begin{equation*}
V_N={\rm span}\left\{1,\cos\theta,\cos(2\theta),\dots,\cos(N\theta),\sin \theta,\sin(2\theta),\dots,\sin(N\theta)\right\} .  
\end{equation*}
Then $J(\Omega)$ is entirely determined by the real vector $\left(a_0,a_1,\dots,a_N,b_1,\dots,b_N\right)$ and the problem is transformed into a classical optimization problem in a real vector space.

Let $J(\Omega)=F(a_0,a_1,\dots,a_N,b_1,\dots,b_N)$, where the domain $\Omega$ is determined by  $r_N(\theta)=a_0+\sum_{i= 1}^{N}a_i  \cos(i\theta)+\sum_{i= 1}^{N}b_i  \sin(i\theta)$. For all $0\leq i\leq N$:
\begin{gather*}
    \frac{\partial F(a_0,a_1,\dots,a_N,b_1,\dots,b_N)}{\partial a_i}
    =\lim_{t \to 0^+}\frac{F(a_0,\dots,t+a_i\dots,a_N,b_1,\dots,b_N)-F(a_0,\dots,a_i\dots,a_N,b_1,\dots,b_N)}{t} , 
\end{gather*}
the boundary function corresponding to $\left(a_0,\dots,t+a_i,\dots,a_N,b_1,\dots,b_N\right)$ is 
$r_N(\theta)+t\cos(i\theta)$.
Therefore, the function of the induced velocity field is $dr(\theta)=\cos(i\theta)$, and the value of the velocity field on the boundary is $V(0,r_N(\widehat{x} )\widehat{x} )=dr(\widehat{x} )\widehat{x}$.
Moreover,
\begin{equation}
    \frac{\partial F(a_0,a_1,\dots,a_N,b_1,\dots,b_N)}{\partial a_i}
    =dJ(\Omega)[V].\label{Partial_a}
\end{equation}
Similarly,
\begin{equation}
    \frac{\partial F(a_0,a_1,\dots,a_N,b_1,\dots,b_N)}{\partial b_i}
    =dJ(\Omega)[V],\label{Partial_b}
\end{equation}
where $V(0,r(\widehat{x} )\widehat{x} )=dr(\widehat{x} )\widehat{x}$, $dr(\theta)=\sin(i\theta)$.

From (\ref{J_Euler}) we see that the Euler derivative involves the curvature of the free boundary which is difficult to calculate. So we propose to avoid it by using the tangential Green's formula. In fact,
\begin{equation*}
    \begin{split}
        dJ(\Omega)[V]
        &=\int_{\Gamma} v_n\left\{-\frac{1}{2}p^2\mathcal{H} 
        +p\left[ f-\mathcal{H} g-\frac{\partial g}{\partial \textbf{n}} \right] 
        -(g+p)\frac{\partial p}{\partial \textbf{n}}\right\} d\sigma \\
          &=\int_{\Gamma}v_n\left\{ 
        p\left[ f-\frac{\partial g}{\partial \textbf{n}} \right] 
        -(g+p)\frac{\partial p}{\partial \textbf{n}}\right\} d\sigma
        -\int_{\Gamma}v_n\left(\frac{1}{2}p^2+pg\right) \mathcal{H} 
          d\sigma \\
          &=\int_{\Gamma}v_n\left\{ 
            p\left[ f-\frac{\partial g}{\partial \textbf{n}} \right] 
            -(g+p)\frac{\partial p}{\partial \textbf{n}}\right\} d\sigma
            -\int_{\Gamma} {\rm div}_{\Gamma}\left(\left(\frac{1}{2}p^2+pg\right)V \right) d\sigma\\
        &=\int_{\Gamma}v_n\left\{ 
            p\left[ f-\frac{\partial g}{\partial \textbf{n}} \right] 
            -(g+p)\frac{\partial p}{\partial \textbf{n}}\right\} d\sigma\\
            &-\int_{\Gamma}
            \left\{{\rm div}\left(\left(\frac{1}{2}p^2+pg\right)V\right) 
            -D\left(\left(\frac{1}{2}p^2+pg\right)V\right) \textbf{n}\cdot \textbf{n}\right\} 
            d\sigma.
    \end{split}
\end{equation*}

In the identity (\ref{Partial_a}), 
according to the analysis above, we know that
the value of $V(0,x)$ at the boundary point $x=(x_1,x_2) \in \Gamma$ is $\cos(i\theta)$, where $\theta$ is the polar angle of the corresponding point $\widehat{x}\in \mathbb{S} ^{n-1}$ of $x\in \Gamma$.
$\theta$ is also the polar angle of $x$ because $x$ and $\widehat{x}$ are collinear. The polar angle of $(x_1,x_2)$ is denoted by $\theta(x_1,x_2)$, and we have
\begin{equation*}
    V(0,x)=\cos(i\theta(x_1,x_2))\binom{\cos \theta(x_1,x_2) }{\sin\theta(x_1,x_2) },\quad x=(x_1,x_2) \in \Gamma    .
\end{equation*}
From the expression of the Euler derivative, $dJ(\Omega)[V]$ is completely determined by the value of $V(0,x)$ on $\Gamma$. In other words, if $U(x)=W(x),\forall x\in \Gamma$, then $dJ(\Omega)[U]=dJ(\Omega)[W]$.
As can be seen from the above, in order to avoid calculating the curvature, we need to know the value of $V(0,x)$ in the neighborhood of $\Gamma$.
Let $U(x)=V(0,x)$, $x\in \Gamma$, we can perform the following harmonic extensions on $V(0,x)=(v_1,v_2)$:
\begin{equation*}
    \begin{cases}
        -\Delta w_1=0\quad {\rm in}\quad \Omega, \\
        w_1=v_1\quad {\rm on}\quad \Gamma ,\\
        w=0\quad {\rm on}\quad\Sigma 
    \end{cases}
    \text{and}\quad
    \begin{cases}
        -\Delta w_2=0\quad {\rm in}\quad \Omega, \\
        w_2=v_2\quad {\rm on}\quad \Gamma ,\\
        w_2=0\quad {\rm on}\quad\Sigma .
        \end{cases}
\end{equation*}
Then $dJ(\Omega)[V]=dJ(\Omega)[W]$, where $W=(w_1,w_2)$.
Meanwhile, $W$ ensures that the inner boundary $\Sigma$ remains fixed during the transformation.

We summarize the formulae for calculating partial derivatives of $F$ as follows:
\begin{equation*}
    \frac{\partial F(a_0,a_1,\dots,a_N,b_1,\dots,b_N)}{\partial a_i}
    =dJ(\Omega)[W],\quad i=0,1,\dots,N,
\end{equation*}
where $\Omega$ is the domain determined by $\left(a_0,a_1,\dots,a_N,b_1,\dots,b_N\right)$
and $W$ is a harmonic extension of
\begin{equation*}
    V(x)=
    \begin{cases}
        \cos(i \theta(x_1,x_2))\binom{\cos \theta(x_1,x_2) }{\sin  \theta(x_1,x_2) } \quad \quad x=(x_1,x_2)\in \Gamma ,\\
        0\quad x=(x_1,x_2)\in \Sigma.
    \end{cases}
\end{equation*}
Similarly, 
\begin{equation*}
    \frac{\partial F(a_0,a_1,\dots,a_N,b_1,\dots,b_N)}{\partial b_i}
    =dJ(\Omega)[W],\quad i=1,\dots,N,
\end{equation*}
where $\Omega$ is domain determined by $\left(a_0,a_1,\dots,a_N,b_1,\dots,b_N\right)$
and $W$ is a harmonic extension of
\begin{equation*}
    V(x)=
    \begin{cases}
        \sin(i \theta(x_1,x_2))\binom{\cos \theta(x_1,x_2) }{\sin  \theta(x_1,x_2)}
        \quad x=(x_1,x_2)\in \Gamma ,\\
        0\quad x=(x_1,x_2)\in \Sigma.
    \end{cases}
\end{equation*}

We use FreeFem++ (\cite{Hecht}) to solve the partial differential equations involved in calculating partial derivatives. With the partial derivative calculation formulae, we can use the gradient descent algorithm or BFGS algorithm to solve the following optimization problem:
\begin{equation*}
\inf_{\left(a_0,a_1,\dots,a_N,b_1,\dots,b_N\right)\in \mathbb{R}^{2N+1}} F\left(a_0,a_1,\dots,a_N,b_1,\dots,b_N\right).
\end{equation*}



Let $\textbf{x}\triangleq \left(a_0,a_1,\dots,a_N,b_1,\dots,b_N\right)$, the gradient descent  algorithm is presented below \cite[Chapter 3, pp. 56-57]{yyx}:

\begin{tabular*}{\textwidth}{l}
\toprule
$$\textbf{Algorithm}\quad$$Gradient descent method \\
\midrule
1.$$\textbf{Initialization}$$:
Choose an initial point $\textbf{x}^0$, an initial step size $\alpha>0$, a step size lower bound $\varepsilon $, the \\
\quad\quad maximum  number of iterations $M$, $k=0$,
$\beta_1,\beta_2 \in(0,1)$, and an auxiliary variable $z=0$.
\\
2. $$\textbf{Process}$$:\\
3.
\quad$$\textbf{for}$$ $k=0,1,2,\dots,M$ $$\textbf{do}$$\\
4.
\quad\quad$z=F(\textbf{x}^k-\alpha\nabla F(\textbf{x}^k))$.\\
5.
\quad\quad $$\textbf{if}$$ $z<F(\textbf{x}^k)$
$$\textbf{then}$$ \\ 
6.
\quad\quad\quad $\alpha=\alpha/\beta_1$. \\
7.
\quad\quad $$\textbf{end if}$$\\
8.
\quad\quad $$\textbf{while}$$ $z\geq F(\textbf{x}^k)$ and $\alpha\geq \varepsilon$ 
$$\textbf{do}$$\\
9.
\quad\quad \quad $\alpha=\alpha*\beta_2$,\\
10.\quad\quad \quad $z=F(\textbf{x}^k-\alpha\nabla F(\textbf{x}^k))$.\\
11.\quad\quad $$\textbf{end}$$
$$\textbf{while}$$\\
12.\quad\quad $$\textbf{if}$$ $\alpha<\varepsilon$
$$\textbf{then}$$ \\
13.\quad\quad\quad break;\\
14.\quad\quad $$\textbf{end if}$$\\
15.\quad\quad $\textbf{x}^{k+1}=\textbf{x}^k-\alpha\nabla F(\textbf{x}^k)$\\
16.\quad$$\textbf{end}$$
$$\textbf{for}$$\\
17. $$\textbf{Output}:$$ The newest $\textbf{x}^{k}$.\\
\bottomrule
\end{tabular*}

\begin{Remark}
Here we choose $\alpha<\varepsilon$ as the stopping criterion for two reasons. First, when the step size is small enough, the optimization variable changes only slightly,  correspondingly the shape will hardly change in the shape optimization procedure. On the other hand, when we choose a relatively large $N$, the norm of the gradient is often very large. In this case, it is not appropriate to use the norm of the gradient as a stopping criterion.
\end{Remark}


We remark that the $\|\cdot \|_{H^1(\Gamma^*)} $-norm coercivity of the shape Hessian $d^2J(\Omega^*)[\cdot,\cdot]$
 will ensure the existence of local optimal solutions for the nonlinear Ritz-Galerkin approximation method and its convergence, as presented in Theorem \ref{thm:DL214}. It is easy to see that the error
    $\|r^*_N-r^* \|_{H^1(\mathbb{S} ^{n-1})}\rightarrow 0$ when $V_N$ approaches to $X$.
    


\subsection{Numerical experiments}
The errors of the nonlinear Ritz-Galerkin approximation method come from two aspects: the first source is the truncation error of the subspace $V_N$ approximating to $X$; the second one is caused by solving optimization problems in the subspace. For the former, since $V_1\subset V_2 \subset \cdots \subset V_N \subset \cdots \subset X$, it depends on the selection of $N$; for the latter, the error of solving optimization problems in subspaces mainly comes from the approximation error of the Euler derivative which in turn depends on the numerical accuracy of $u,w,p$ and the velocity field $V$. Since we use finite element methods to solve the PDEs, the latter error depends on the mesh size of the partition of domains.

We conduct numerical experiments from these two perspectives to assess the quality of the numerical solution. According to Remark \ref{remark:rmk322}, we choose $f\equiv0$, $g\equiv3$, $h\equiv1$. In addition, we choose the fixed boundary as the boundary of $(-0.25,0.25)^2\setminus (0,0.25)^2$ and choose the initial free boundary as 
\begin{equation*}
    \Gamma=\left\{\left(\frac{2}{3}+\frac{1}{12}\cos(3t) \right)\binom{\cos t}{\sin t} :\quad t \in [0,2\pi] \right\} .
\end{equation*}
We refer to Fig. \ref{fig:chushiquyu1} for an illustration of the initial domain and its triangulation. In the experiment, we choose the parameters $\alpha=0.005$, $\beta_1=\frac{2}{3}, \beta_2=\frac{1}{2}$, $\varepsilon=10^{-4}$ and $M=200$.

\begin{figure}[!htbp]
    \centering
    \includegraphics[trim = 1mm 2mm 1mm 1mm, clip, width=0.50\textwidth]{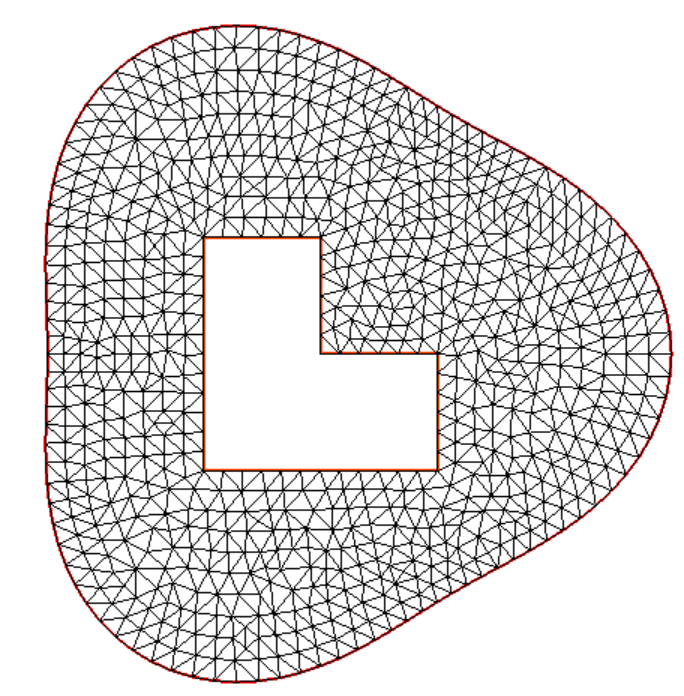}
    \caption{The initial domain and its triangulation for Experiment 1.}
    \label{fig:chushiquyu1}
\end{figure}


In the first group of experiments, the mesh size remains unchanged. We always control the number of mesh points on the free and fixed boundaries to be 100 and 48, respectively. We gradually increase the spatial dimension $N$ to observe the change of errors. Herein, the error refers to the value of the objective functional $J$, which can be regarded as the error of tracking Dirichlet data at the free boundary.
The experimental results are shown in Table \ref{tab:biao1} where NoI denotes the number of iterations. The convergence histories of the errors are shown in Fig. \ref{fig:wcqx} for different $N$, where plot (a) and plot (b) use different vertical axes; the final shapes are presented in Fig. \ref{fig:zzxz}.

\begin{table}[h]
\centering
    \caption{The change of errors with respect to the dimension $N$ for Experiment 1.}
    \label{tab:biao1}
\begin{tabular}{|c|c|c|c|c|}
    \hline
    N&dimension&initial error&NoI&final error\\
    \hline
    3&7&4.60377&15& 0.035097700 \\
    \hline
    4&9&4.60377&15&0.006114220\\
    \hline
    5&11&4.60377&15&0.001840740\\
    \hline
    6&13&4.60377&17&0.001218670\\
    \hline
    7&15&4.60377&28& 0.001033880 \\
    \hline
    8&17&4.60377&29&0.000846881\\
    \hline
    9&19&4.60377&36&0.000828824\\
    \hline
    10&21&4.60377&40&0.000854410\\
    \hline
    11&23&4.60377&50&0.000785928\\
    \hline
\end{tabular}
\end{table}

It can be observed from Table \ref{tab:biao1} that in the first few steps, the error can be significantly reduced when the spatial dimension is increased. When $N\geq 8$, the error is basically stable at around 0.0008. This is because, in addition to the inherent error brought by subspace approximations, finite element solutions of $u,w,p,V$ will also cause errors. In addition, geometric approximation errors appear because we use polygons to approximate smooth domains.

\begin{figure}[!htbp]
    \centering
    \includegraphics[trim = 1mm 2mm 1mm 1mm, clip, width=0.9\textwidth]{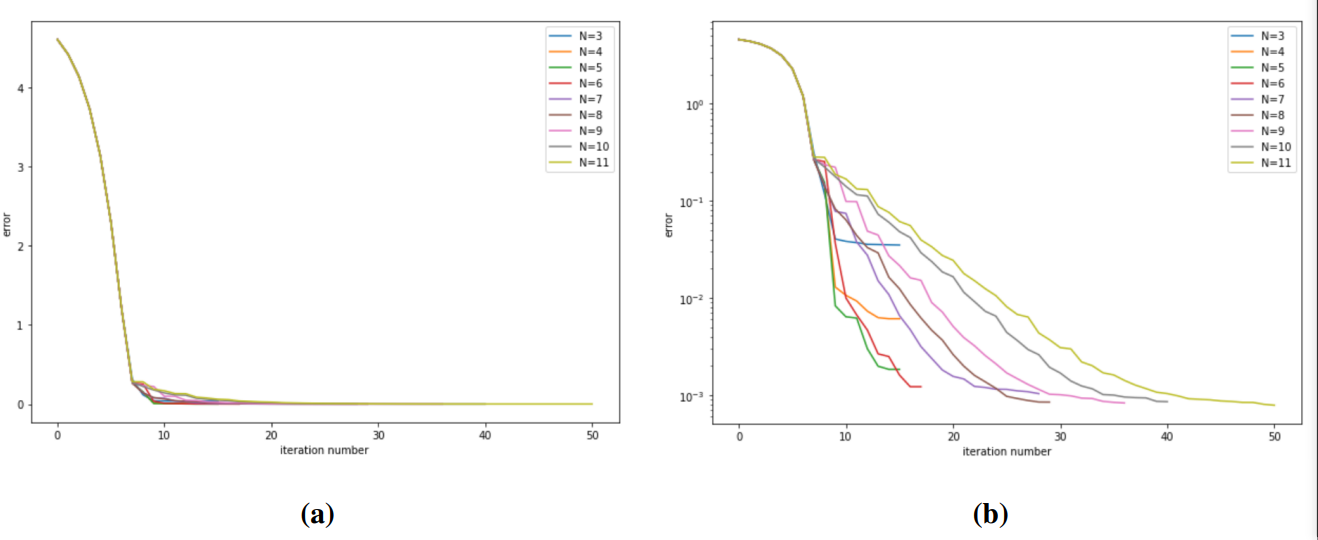}
    \caption{The convergence histories of errors with respect to the iteration number and $V_N$ for Experiment 1.}
    \label{fig:wcqx}
\end{figure}

\begin{figure}[!htbp]
    \centering
    \includegraphics[trim = 1mm 2mm 1mm 1mm, clip, width=0.9\textwidth]{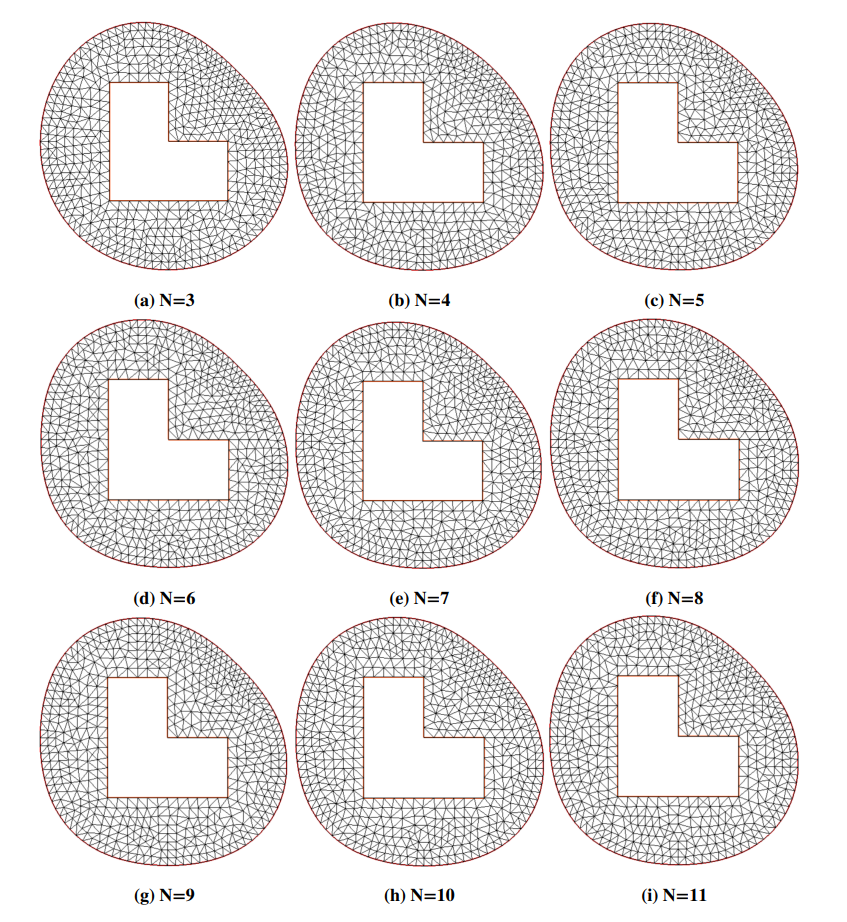}
    \caption{The final shapes obtained from different choices of $N$ for Experiment 1.}
    \label{fig:zzxz}
\end{figure}

In the second group of experiments we investigate the error behavior by using coarser and finer meshes than those used in the first group of experiments for $N=8$. The number of mesh points on the fixed and free boundaries is denoted by cnt1 and cnt2, respectively. The experimental results are shown in Table \ref{tab:biao2}.

\begin{table}[h]
    \centering
    \caption{The change of errors with respect to mesh levels for Experiment 2.}
    \label{tab:biao2}
\begin{tabular}{|c|c|c|c|c|}
    \hline
    cnt1&cnt2&initial error&NoI&final error\\
    \hline
    24&50& 4.51939&23 &0.004479200 \\
    \hline
   48&100&4.60377&29&0.000846881\\
    \hline
96&200&4.63436&29& 0.000702069\\
    \hline
    192&400&4.64579 &33& 0.000777745 \\
    \hline
\end{tabular}
\end{table}

According to Table \ref{tab:biao2}, reducing the mesh size can reduce the error overall, but it can also increase the size of the finite element stiffness matrix and computational cost, and decrease the computational accuracy that may not necessarily reduce errors. For example, the error has increased from cnt1=96, cnt2=200 to cnt1=192, cnt2=400. Another important reason is that we can only guarantee that the mesh size is smaller at the beginning, but the mesh size may be enlarged during the shape optimization procedure.


The third group of experiments is independent of the first two groups. Herein, an over-determined equation with a known exact solution $u$ and domain $\Omega^*$ is constructed to verify the performance of the algorithm. 
It can be verified that if the free boundary of the domain $\Omega^*$ is a unit circular surface with center $(x_0,y_0)$ and its fixed boundary is arbitrary, then it ensures that the over-determined equation
\begin{equation*}
    \begin{cases}
    -\Delta u =0\quad \rm{in} \quad \Omega^*,\\
   -\frac{\partial u}{\partial \textbf{n}}=C,u=0 \quad \rm{on} \quad \Gamma^*,\\
    u=-C\log\sqrt{(x-x_0)^2+(y-y_0)^2} \quad \rm{on} \quad \Sigma 
    \end{cases}
\end{equation*}
 admits a solution, given by $u=-C\log \sqrt{(x-x_0)^2+(y-y_0)^2}$.
 Next, we choose $C=1$, $x_0=0$, $y_0=0$ to verify the effectiveness of our algorithm. 
 
Since the free boundary of $\Omega^*$ is a circular surface centered at the origin, $r^*\in V_1$. We will investigate the effect of numerical computations by changing the mesh size. We choose the initial free boundary as
\begin{equation*}
    \Gamma=\left\{\left(\frac{2}{3}+\frac{1}{12}\sin(t) \right)\binom{\cos t}{\sin t} :\quad t \in [0,2\pi] \right\}, 
\end{equation*}
the fixed boundary is selected as a circle $\{x^2+y^2=0.3^2\}$.
The number of mesh points on the fixed boundary is denoted by cnt1, and that on the free boundary is denoted by cnt2.
The experimental results are shown in Table \ref{tab:biao3}, while Fig. \ref{fig:bsantu} shows the initial shape and the final shape with cnt1=160,  cnt2=400. As can be seen from Table \ref{tab:biao3}, we observe the similar convergence property as that of the second group of experiments.

\begin{table}[h]
    \centering
    \caption{The change of errors with respect to mesh levels for Experiment 3.}
    \label{tab:biao3}
\begin{tabular}{|c|c|c|c|c|}
    \hline
    cnt1&cnt2&initial error&NoI&final error\\
    \hline
    20&50& 4.16158&14 &0.000518986 \\
    \hline
    40&100& 4.14181&15&0.000512769\\
    \hline
    80&200&4.13028&20& 0.000486740\\
    \hline
    160&400& 4.12340 &27& 0.000545819 \\
    \hline
\end{tabular}
\end{table}

\begin{figure}[!htbp]
    \centering
    \includegraphics[trim = 1mm 2mm 1mm 1mm, clip, width=0.7\textwidth]{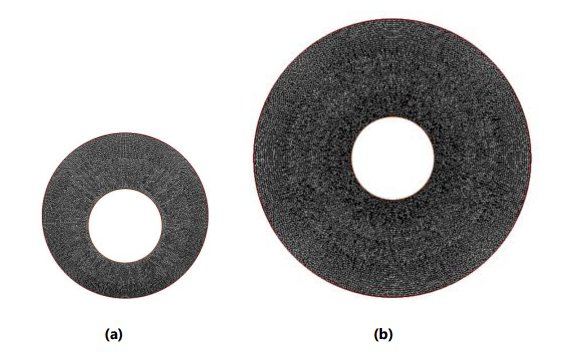}
    \caption{The initial and final shapes for Experiment 3.}
    \label{fig:bsantu}
\end{figure}




\begin{thebibliography}{99}

\bibitem{allaire2021shape}
G. Allaire, C. Dapogny and F. Jouve, Shape and Topology Optimization, in {\it Geometric Partial Differential Equations}, edited by Bonito A. et al., Elsevier, Amsterdam,  1–132, 2021.

\bibitem{Apel}
T. Apel, O. Steinbach and M. Winkler, Error estimates for Neumann boundary control problems with energy regularization. J. Numer. Math., 24(2016), no. 4, 207-233.

\bibitem{bach1999variational}
M.F. Bach and M. Flucher, Variational Problems with Concentration. Springer Science \&
Business Media, 1999.

\bibitem{Boulkhemair}
A. Boulkhemair, A. Chakib, A. Nachaoui, A.A. Niftiyev and A. Sadik, On a numerical shape optimization approach for a class of free boundary problems. Comput. Optim. Appl., 77(2020), no. 2, 509–537. 

\bibitem{Burman}
E. Burman, D. Elfverson, P. Hansbo, M.G. Larson and K. Larsson, A cut finite element
method for the Bernoulli free boundary value problem. Comput. Methods Appl. Mech.
Engrg., 317(2017), 598-618.
 
\bibitem{Brugger}
R. Br\"ugger, R. Croce and H. Harbrecht, Solving a Bernoulli type free boundary problem with random diffusion. ESAIM Control Optim. Calc. Var., 26(2020), Paper No. 56, 16 pp. 
  
\bibitem{caffarelli1981existence}
L.A Caffarelli and H.W. Alt, Existence and regularity for a minimum problem with free boundary. J. Reine Angew. Math., 325(1981), 105-144.


\bibitem{dambrine2002variations}
M. Dambrine, On variations of the shape Hessian and sufficient conditions for the stability of
critical shapes. RACSAM. Rev. R. Acad. Cienc. Exactas Fís. Nat. Ser. A Mat., 96(2002), 95-121.

\bibitem{dambrine2000stability}
M. Dambrine and M. Pierre, About stability of equilibrium shapes. ESAIM: Math.
Model. Numer. Anal., 34(2000), 811-834.

\bibitem{delfour2011shapes}
M.C. Delfour and J.P. Zol\'esio, Shapes and Geometries: Metrics, Analysis, Differential Calculus, and Optimization. SIAM, Philadelphia, PA, 2011.

\bibitem{EpplerCC}
K. Eppler, Second derivatives and sufficient optimality conditions for shape functionals. Control Cybernet., 29(2000), no. 2, 485–511.

\bibitem{eppler2000optimal}
K. Eppler, Optimal shape design for elliptic equations via BIE-methods. Int. J. Appl. Math. Comput. Sci., 10(2000), 487-516.

\bibitem{eppler2012kohn}
K. Eppler and H. Harbrecht, On a Kohn-Vogelius like formulation of free boundary problem. Comput. Optim. Appl., 52(2012), 69-85.

\bibitem{leugering2012constrained}
K. Eppler and H. Harbrecht, Shape optimization for free boundary problems-analysis and numerics, in {\it Constrained Optimization and Optimal Control for Partial Differential Equations}, edited by  Leugering G. et al., Inter. Series Numer. Math., Birkh\"auser Basel, 277-288, 2012.

\bibitem{eppler2010tracking1}
K. Eppler and H. Harbrecht, Tracking Neumann data for stationary free boundary problems. SIAM J. Control Optim., 48(2010), 2901-2916.

\bibitem{eppler2010tracking2}
K. Eppler and H. Harbrecht, Tracking Dirichlet data in $L^2$ is an ill-posed problem. J. Optim. Theory Appl., 145(2010), 17-35.

\bibitem{eppler2006efficient}
K. Eppler and H. Harbrecht, Efficient treatment of stationary free boundary problems. Appl. Numer. Math., 56(2006), 1326-1339.


\bibitem{eppler2007convergence}
K. Eppler, H. Harbrecht and R. Schneider, On convergence in elliptic shape optimization. SIAM J. Control Optim., 46(2007), no. 1, 61-83.

\bibitem{Fumagalli}
I. Fumagalli, N. Parolini and M. Verani, Shape optimization for Stokes flows: a finite element convergence analysis. ESAIM Math. Model. Numer. Anal., 49(2015), no. 4, pp. 921-951.

\bibitem{Gong}
W. Gong, M. Mateos, John R. Singler and Y. Zhang, Analysis and approximations of Dirichlet boundary control of Stokes flows in the energy space. SIAM J. Numer. Anal., 60(2022), no. 1, 450-474. 

\bibitem{harbrecht2008newton}
H. Harbrecht, A Newton method for Bernoulli's free boundary problem in three dimensions. Computing, 82(2008), 11-30.

\bibitem{Haslinger}
J. Haslinger, K. Ito, T. Kozubek, K. Kunisch and G. Peichl, On the shape derivative for problems of Bernoulli type. Interfaces Free Bound., 11(2009), no. 2, 317–330.

\bibitem{Hecht}
F. Hecht, New development in FreeFem++. J. Numer. Math., 20(2012), no. 3-4, 251–265.
 
\bibitem{Kiniger}
B. Kiniger and B. Vexler, A priori error estimates for finite element discretizations of a shape optimization problem. ESAIM Math. Model. Numer. Anal., 47(2013), pp. 1733-1763.

\bibitem{kohn1984determining}
R. Kohn and M. Vogelius, Determining conductivity by boundary measurements. Commun. Pure  Appl. Math., 37(1984), 289-298.

\bibitem{Rabago}
J.F.T. Rabago and H. Azegami, A second-order shape optimization algorithm for solving the exterior Bernoulli free boundary problem using a new boundary cost functional. Comput. Optim. Appl., 77(2020), no. 1, 251–305. 

 
\bibitem{yyx}
LE. Scales, Introduction to Non-Linear Optimization. Macmillan, 1985.

\bibitem{sokolowski1992introduction}
J. Sokolowski and J.P. Zolésio, Introduction to Shape Optimization. Springer-Verlag, Berlin, 1992.



\end{thebibliography}
\end{document}